\documentclass[oneside,american]{amsart}
\usepackage[T1]{fontenc}
\usepackage[latin9]{inputenc}
\usepackage{babel}
\usepackage{mathrsfs}
\usepackage{amstext}
\usepackage{amsthm}
\usepackage{amssymb}
\usepackage[all]{xy}
\usepackage[unicode=true,pdfusetitle,
 bookmarks=true,bookmarksnumbered=false,bookmarksopen=false,
 breaklinks=false,pdfborder={0 0 1},backref=false,colorlinks=false]
 {hyperref}

\makeatletter
\numberwithin{equation}{section}
\numberwithin{figure}{section}
\theoremstyle{plain}
\newtheorem{thm}{\protect\theoremname}
  \theoremstyle{remark}
  \newtheorem{rem}[thm]{\protect\remarkname}
  \theoremstyle{plain}
  \newtheorem{lem}[thm]{\protect\lemmaname}
  \theoremstyle{definition}
  \newtheorem{defn}[thm]{\protect\definitionname}
\newenvironment{lyxlist}[1]
{\begin{list}{}
{\settowidth{\labelwidth}{#1}
 \setlength{\leftmargin}{\labelwidth}
 \addtolength{\leftmargin}{\labelsep}
 }}
{\end{list}}
  \theoremstyle{plain}
  \newtheorem{prop}[thm]{\protect\propositionname}

\newcommand{\xyR}[1]{\xydef@\xymatrixrowsep@{#1}}
\newcommand{\xyC}[1]{\xydef@\xymatrixcolsep@{#1}}

\newcommand{\Min}{\mathrm{Min}}
\newcommand{\Fi}{\mathrm{Fil}}
\newcommand{\Gr}{\mathrm{Gr}}
\newcommand{\Frac}{\mathrm{Frac}}

\newcommand{\Aut}{\mathrm{Aut}}

\newcommand{\Gal}{\mathrm{Gal}}
\newcommand{\Hom}{\mathrm{Hom}}

\newcommand{\Gra}{\mathsf{Gr}}
\newcommand{\Fil}{\mathsf{Fil}}
\newcommand{\Rep}{\mathsf{Rep}}

\newcommand{\Vect}{\mathsf{Vect}}

\newcommand{\Iso}{\mathsf{Iso}}

\makeatother

  \providecommand{\definitionname}{Definition}
  \providecommand{\lemmaname}{Lemma}
  \providecommand{\propositionname}{Proposition}
  \providecommand{\remarkname}{Remark}
\providecommand{\theoremname}{Theorem}

\begin{document}

\title{Mazur's inequality and Laffaille's theorem}

\author{Christophe Cornut}
\begin{abstract}
We look at various questions related to filtrations in $p$-adic Hodge
theory, using a blend of building and Tannakian tools. Specifically,
Fontaine and Rapoport used a theorem of Laffaille on filtered isocrystals
to establish a converse of Mazur's inequality for isocrystals. We
generalize both results to the setting of (filtered) $G$-isocrystals
and also establish an analog of Totaro's $\otimes$-product theorem
for the Harder-Narasimhan filtration of Fargues.
\end{abstract}



\address{CNRS, Sorbonne Université, Université Paris Diderot, Institut de
Mathématiques de Jussieu-Paris Rive Gauche, IMJ-PRG, F-75005, Paris,
France.}

\email{\href{mailto:christophe.cornut@imj-prg.fr}{christophe.cornut@imj-prg.fr}}

\urladdr{\href{https://www.imj-prg.fr/~christophe.cornut/}{https://www.imj-prg.fr/$\sim$christophe.cornut/}}
\maketitle

\section{Introduction}

Many spaces that show up in $p$-adic Hodge theory are closely related
to buildings, if only because they involve filtrations and lattices
(or norms), which are respectively classified by Tits and Bruhat-Tits
buildings. Quite naturally, building theoretical tools have thus occasionally
been used to study these spaces, as in~\cite{RaZi02}, \cite{VoWe11}
or \cite{ChVi18}. In continuation with \cite{CoNi16}, this paper
tries to promote the idea that a more systematic use of buildings
sometimes leads to streamlined proofs and new results in $p$-adic
Hodge theory, and certainly already provides an enlightening new point
of view on various classical notions, simply by recasting them in
a metric geometry framework. 

For instance in \cite{CoNi16}, we showed that the Newton decomposition
of an isocrystal $(V,\varphi)$ over an algebraically closed residue
field $k$ of characteristic $p>0$ can be identified with the translation
vector of $\varphi$, viewing $\varphi$ as a semi-simple isometry
of the extended Bruhat-Tits building of $G=GL(V)$ \textendash{} this
still holds when $k$ is merely perfect. Filtrations $\mathcal{F}$
on $V$ come into the picture in the metric guise of non-expanding
maps on the building, and the weak admissibility of a filtered isocrystal
$(V,\varphi,\mathcal{F})$ is related to the joint dynamic of the
$(\varphi,\mathcal{F})$-actions on the building: this was somehow
first observed by Laffaille \cite{La00}, who showed that $(V,\varphi,\mathcal{F})$
is weakly admissible if and only if $V$ contains a strongly divisible
lattice, giving rise to a special point of the building fixed by the
composition of $\varphi$ and $\mathcal{F}$. This criterion for weak
admissibility implies that $V$ contains a lattice $L$ such that
$L$ and $\varphi L$ are in a given relative position $\mu$ if and
only if there is a weakly admissible filtration $\mathcal{F}$ of
type $\mu$ on $V$. Assuming that the residue field $k$ is algebraically
closed, Fontaine and Rapoport \cite{FoRa05} gave a criterion for
the existence of such an $\mathcal{F}$, thereby establishing a converse
to the already known necessary condition for the existence of $L$
\textendash{} Mazur's inequality. Our first target is the generalization
of these results to more general reductive groups $G$ (Theorem~\ref{thm:Main1}
and \ref{thm:Main2MazurRec}). Note that the converse of Mazur's inequality
for unramified groups was already established by \cite{Ga10}, following
an entirely different strategy suggested by \cite{Ko03}. Our method
yields a somewhat weaker result: we say nothing about the existence
of \emph{hyperspecial }(strongly divisible) points.

Here is a rough dictionary that may help the reader digest the previous
paragraph and warm-up for the sequel. Norms on $V$ correspond to
points on the (extended) Bruhat-Tits building $X$ of $G$, and lattices
correspond to a $G$-orbit $X^{\circ}$ of hyperspecial points in
$X$. There is a canonical $G$-equivariant distance $d$ on $X$
with nice convexity properties: $(X,d)$ is a complete CAT(0)-space,
of which $X^{\circ}$ is a closed discrete subset. The $\mathbb{R}$-graduations
$\mathcal{G}$ and (non-increasing) $\mathbb{R}$-filtrations $\mathcal{F}$
on $V$ correspond to asymptotic classes of, respectively, (constant
speed) geodesic lines $\ell:\mathbb{R}\rightarrow X$, and geodesic
rays $c:\mathbb{R}_{+}\rightarrow X$. In particular, $\mathbb{R}$-filtrations
act on $X$ by non-expanding maps as follows: $x+\mathcal{F}=c(1)$,
where $c$ is the unique geodesic ray in $\mathcal{F}$ emanating
from $x\in X$. If $x\in X^{\circ}$ corresponds to a lattice $L$
in $V$ and $\mathcal{F}$ is a $\mathbb{Z}$-filtration, then $x+\mathcal{F}$
also belongs to $X^{\circ}$ and corresponds to the lattice
\[
L+\mathcal{F}=\sum_{i\in\mathbb{Z}}p^{-i}L\cap\mathcal{F}^{i}.
\]
The $G$-orbits of $\mathbb{R}$-filtrations on $V$ are classified
by their type $t(\mathcal{F})\in\mathbb{R}_{\geq}^{r}$, where 
\[
\mathbb{R}_{\geq}^{r}=\{(x_{1},\cdots,x_{r}):x_{1}\geq\cdots\geq x_{r}\}
\]
is usually identified with the set of concave polygons over $[0,r]$
whose brake points have integral $x$-coordinate. We equip $\mathbb{R}_{\geq}^{r}$
with its usual partial order, given by 
\[
(x_{1},\cdots,x_{r})\leq(x_{1}^{\prime},\cdots,x_{r}^{\prime})\iff\begin{cases}
\sum_{i=1}^{j}x_{i}\leq\sum_{i=1}^{j}x_{i}^{\prime} & \forall j\in\{1,\cdots,r-1\}\\
\sum_{i=1}^{r}x_{i}=\sum_{i=1}^{r}x_{i}^{\prime}.
\end{cases}
\]
We set $\mathbb{Q}_{\geq}^{r}=\mathbb{R}_{\geq}^{r}\cap\mathbb{Q}^{r}$
and $\mathbb{Z}_{\geq}^{r}=\mathbb{R}_{\geq}^{r}\cap\mathbb{Z}^{r}$.
The formula $\mathbf{d}(x,x+\mathcal{F})=t(\mathcal{F})$ defines
a $G$-invariant vector-valued distance $\mathbf{d}:X\times X\rightarrow\mathbb{R}_{\geq}^{r}$,
whose composition with the standard euclidean norm on $\mathbb{R}_{\geq}^{r}\subset\mathbb{R}^{r}$
is equal to the canonical distance $d$ of $X$. The restriction of
$\mathbf{d}$ to $X^{\circ}$ yields a bijection $G\backslash\left(X^{\circ}\times X^{\circ}\right)\simeq\mathbb{Z}_{\geq}^{r}$,
which is the usual invariant describing the relative position of two
lattices in $V$. 

The Frobenius $\varphi$ of $V$ induces an isometry $\varphi$ of
$(X,d)$ which preserves $X^{\circ}$. It is a semi-simple isometry,
which means that if $\min(\varphi)=\inf\left\{ d(\varphi x,x):x\in X\right\} $,
then 
\[
\Min(\varphi)=\left\{ x\in X:d(\varphi x,x)=\min(\varphi)\right\} 
\]
is \emph{non-empty. }This is a\emph{ }closed, convex, $\varphi$-stable
subset of $X$, equal to the disjoint union of the $\varphi$-stable
geodesic lines of $X$. The common asymptotic class of these lines
corresponds to the Newton $\mathbb{Q}$-graduation $\mathcal{G}_{N}$
of $(V,\varphi)$. The induced pair of opposed Newton $\mathbb{Q}$-filtrations
\emph{$(\mathcal{F}_{N},\mathcal{F}_{N}^{\iota})$} on $V$ act on
$x\in\Min(\varphi)$ as follows: 
\[
x+\mathcal{F}_{N}=\varphi^{-1}(x)\quad\text{and}\quad x+\mathcal{F}_{N}^{\iota}=\varphi(x).
\]
In particular, the Newton type $t_{N}=t(\mathcal{F}_{N})\in\mathbb{Q}_{\geq}^{r}$
is equal to $\mathbf{d}(\varphi(x),x)$ for every $x\in\Min(\varphi)$,
and $\min(\varphi)=\left\Vert t_{N}\right\Vert $. Moreover, we have
\[
\Min(\varphi)=X_{\varphi}(t_{N})=X_{\varphi}(\mathcal{F}_{N})
\]
where for any $\mu\in\mathbb{R}_{\geq}^{r}$ and any $\mathbb{R}$-filtration
$\mathcal{F}$ on $V$ of type $t(\mathcal{F})=\mu$, 
\begin{align*}
X_{\varphi}(\mu) & =\left\{ x\in X:\mathbf{d}(\varphi x,x)=\mu\right\} ,\\
X_{\varphi}(\mathcal{F}) & =\left\{ x\in X:x=\varphi x+\mathcal{F}\right\} .
\end{align*}
By definition, $X_{\varphi}(\mu)$ is a closed $\Aut(V,\varphi)$-stable
subset of $X$, $X_{\varphi}(\mathcal{F})$ is a closed convex $\Aut(V,\varphi,\mathcal{F})$-stable
subset of $X$, and we have a covering 
\begin{equation}
X_{\varphi}(\mu)=\cup_{t(\mathcal{F})=\mu}X_{\varphi}(\mathcal{F}).\label{eq:CoveringAFFDEL}
\end{equation}
For $\mu\in\mathbb{Z}_{\geq}^{r}$ and any $\mathbb{Z}$-filtration
$\mathcal{F}$ on $V$ of type $t(\mathcal{F})=\mu$, we set
\[
X_{\varphi}^{\circ}(\mu)=X_{\varphi}(\mu)\cap X^{\circ}\quad\text{and}\quad X_{\varphi}^{\circ}(\mathcal{F})=X_{\varphi}(\mathcal{F})\cap X^{\circ}
\]
so that $X_{\varphi}^{\circ}(\mu)$ is stable under $\Aut(V,\varphi)$,
$X_{\varphi}^{\circ}(\mathcal{F})$ is stable under $\Aut(V,\varphi,\mathcal{F})$,
and 
\begin{equation}
X_{\varphi}^{\circ}(\mu)=\cup_{t(\mathcal{F})=\mu}X_{\varphi}^{\circ}(\mathcal{F}).\label{eq:CovAFFDELLHYP}
\end{equation}
The subsets $X_{\varphi}^{\circ}(\mu)$ are usually called affine
Deligne-Lusztig ``varieties'', and there is an extensive literature
about them (and various generalizations): they show up in the description
of special fibers of Shimura varieties, and indeed can be equipped
with some sort of algebraic structure over the relevant residue field.
The covering subsets $X_{\varphi}^{\circ}(\mathcal{F})$ are made
of strongly divisible lattices in the sense of \cite{Laf80,FoLa82},
and they are related to stable lattices in crystalline Galois representations,
but the covering (\ref{eq:CovAFFDELLHYP}) itself has not attracted
much attention, even though it was crucially used in \cite{FoRa05}
to obtain a converse to Mazur's inequality using Laffaille's criterion
for weak admissibility, as we shall now explain. 

In this simple setting where $G=GL(V)$, Mazur's inequality and Laffaille's
criterion can be stated as follows: for any $\mu\in\mathbb{Z}_{\geq}^{r}$
and any $\mathbb{Z}$-filtration $\mathcal{F}$ on $V$,
\begin{equation}
X_{\varphi}^{\circ}(\mu)\neq\emptyset\Longrightarrow\mu\geq t_{N}\text{ in }\mathbb{R}_{\geq}^{r},\label{eq:MazurIneq1}
\end{equation}
\begin{equation}
X_{\varphi}^{\circ}(\mathcal{F})\neq\emptyset\iff(V,\varphi,\mathcal{F})\text{ is (weakly) admissible.}\label{eq:LaffailleCrit1}
\end{equation}
Let $\mathbf{Adm}_{\varphi}(\mu)$ be the set of all (weakly) admissible
$\mathbb{Z}$-filtrations $\mathcal{F}$ of type $t(\mathcal{F})=\mu$
on $(V,\varphi)$. Using (\ref{eq:CovAFFDELLHYP}) and (\ref{eq:LaffailleCrit1}),
the converse of Mazur's inequality (\ref{eq:MazurIneq1}) becomes
\begin{equation}
\mu\geq t_{N}\text{ in }\mathbb{R}_{\geq}^{r}\Longrightarrow\mathbf{Adm}_{\varphi}(\mu)\neq\emptyset.\label{eq:FontRapp}
\end{equation}
This existence result is established by Fontaine and Rapoport in \cite{FoRa05}:
they show that any sufficiently generic $\mathbb{Z}$-filtration of
type $\mu$ with $\mu\geq t_{N}$ is weakly admissible. 

Forgetting hyperspecial points, let us now explain the results that
we will eventually generalize in the first part of this paper \textendash{}
from $GL(V)$ to an arbitrary but unramified reductive group $G$.
First, we have the following variants of Mazur's inequality and Laffaille's
criterion: for any $\mu\in\mathbb{R}_{\geq}^{r}$ and any $\mathbb{R}$-filtration
$\mathcal{F}$ on $V$,
\begin{equation}
X_{\varphi}(\mu)\neq\emptyset\Longrightarrow\mu\geq t_{N}\text{ in }\mathbb{R}_{\geq}^{r},\label{eq:MazurIneq2}
\end{equation}
\begin{equation}
X_{\varphi}(\mathcal{F})\neq\emptyset\iff(V,\varphi,\mathcal{F})\text{ is weakly admissible.}\label{eq:LaffailleCrit2}
\end{equation}
Note that weak admissibility makes perfect sense for $\mathbb{R}$-filtrations.
Let $\mathbf{Adm}_{\varphi}(\mu)$ be the set of all weakly admissible
$\mathbb{R}$-filtrations $\mathcal{F}$ of type $t(\mathcal{F})=\mu$
on $(V,\varphi)$. Using (\ref{eq:CoveringAFFDEL}) and (\ref{eq:LaffailleCrit2})
as above, the converse of Mazur's inequality (\ref{eq:MazurIneq2})
becomes
\begin{equation}
\mu\geq t_{N}\text{ in }\mathbb{R}_{\geq}^{r}\Longrightarrow\mathbf{Adm}_{\varphi}(\mu)\neq\emptyset.\label{eq:FontRapp2}
\end{equation}
As in \cite{FoRa05}, we establish the latter existence result by
showing that any sufficiently generic $\mathbb{R}$-filtration of
type $\mu$ with $\mu\geq t_{N}$ is weakly admissible, with a proof
now based on a mixture of algebraic geometry and building theoretical
tools.
\begin{rem}
The Fontaine-Rapoport method sketched above amounts to reduce a relatively
hard existence problem in a CAT(0)-space (the converse of (\ref{eq:MazurIneq1})
or (\ref{eq:MazurIneq2})) to an easier existence problem on the boundary
of that space ((\ref{eq:FontRapp}) or (\ref{eq:FontRapp2})), using
a suitable fixed point theorem ((\ref{eq:LaffailleCrit1}) or (\ref{eq:LaffailleCrit2})).
This line of thought has been used in other problems unrelated to
$p$-adic Hodge theory, see \cite{Ka06} or the appendix of \cite{Ku14}. 
\end{rem}
Various criteria for weak admissibility of unramified filtered $G$-isocrystals
are listed in section \ref{sec:Laffaille}, including our generalization
of Laffaille's criterion (the equivalence $(1)\iff(3)$ of Theorem~\ref{thm:Main1}).
Mazur's inequality is addressed in section~\ref{sec:Mazur} (Theorem~\ref{thm:Main2MazurRec}).
In the ramified case, we establish another list of criteria for weak
admissibility in section~\ref{sec:EquivForWeaklyAdm}, using a Harder-Narasimhan
filtration and its compatibility with tensor products. There are now
many proofs of this compatibility, which essentially says that the
tensor product of weakly admissible filtered isocrystals is weakly
admissible. In the unramified case, this was an immediate consequence
of Laffaille's criterion. Faltings \cite{Fa95} generalized Laffaille's
proof to the unramified case, and Totaro \cite{To96} found yet another
proof related to Geometric Invariant Theory. It turns out that Totaro's
proof has a very clean building theoretical translation. Instead of
just repeating his arguments in this smoother framework, we chose
to treat the slightly more difficult case of semi-stable weakly admissible
filtered isocrystals in section~\ref{sec:FarguesFilt} (Theorem~\ref{thm:TensorProductFargues}),
where semi-stability here refers to the notion introduced by Fargues
in connection with his own Harder-Narasimhan filtrations~\cite{Fa12}.
In particular, we prove that Fargues's filtration on weakly admissible
filtered isocrystals is compatible with tensor products (Theorem~\ref{thm:TensorProductFargues}),
and moreover show that it is given by a convex projection (Lemma~\ref{lem:CaractWeaklAdmFiltVectorCase}
and Proposition~\ref{prop:FarguesIsConvProj}): the algebraically
defined Fargues $\mathbb{Q}$-filtration $\mathcal{F}_{F}$ of a weakly
admissible filtered isocrystal $(V,\varphi,\mathcal{F})$ is the best
approximation (in a metric sense) of the Newton $\mathbb{Q}$-filtration
$\mathcal{F}_{N}^{\iota}$ of $(V,\varphi)$ by a filtration whose
steps are weakly admissible subspaces of $(V,\varphi,\mathcal{F})$.

In sections~\ref{sec:Laffaille} and \ref{sec:Mazur}, our filtered
isocrystals are unramified (i.e.~$K=K_{0}$ in standard notations).
In the remaining sections, they are defined with respect to an extension
$L$ of the fraction field $K$ of the Witt vectors $W(k)$, which
we do not require to be finite or unramified. The residue field $k$
is a perfect field of characteristic $p>0$, which is only required
to be algebraically closed in Theorem~\ref{thm:Main2MazurRec}. Throughout
the paper, we work with $\Gamma$-filtrations (defined in~\ref{subsec:GrFil}),
where $\Gamma$ is a non-trivial subgroup of $\mathbb{R}$. For applications
to $p$-adic Hodge theory, one should take $\Gamma=\mathbb{Z}$ for
Hodge filtrations and $\Gamma=\mathbb{Q}$ for Newton, Harder-Narasimhan
and Fargues filtrations. 

The need to properly identify $\mathbb{R}$-filtrations with non-expanding
operators on Bruhat-Tits buildings led us to write \cite{Co14}, which
was initially meant to be the first part of this paper, but eventually
grew way out of proportion. This will be our general background reference
for everything pertaining to $\Gamma$-graduations, $\Gamma$-filtrations
and their types. We similarly refer to \cite{Ko85,RaRi96} for $G$-isocrystals,
\cite{RaZi96,DaOrRa10} for $\Gamma$-filtered $G$-isocrystals, \cite{BrHa99}
for CAT(0)-spaces and \cite{Ti79,Co14} for Bruhat-Tits buildings.

\section{Unramified Filtered Isocrystals\label{sec:Laffaille}}

In this section, we first review various basic notions related to
graduations and filtrations~(\ref{subsec:GrFil}), isocrystals (\ref{subsec:Iso}),
and unramified filtered isocrystals (\ref{subsec:IsoFil}). We then
add $G$-structures to them using the Tannakian framework (\ref{subsec:FilIsoG}),
and gently shift from the latter to the building framework (\ref{subsec:FilIsoGRep}-\ref{subsec:BuildGL})
using \cite{Co14} for the translation. Our main result is Theorem~\ref{thm:Main1},
in particular the equivalence $(1)\iff(3)$ relating weak admissibility
of an unramified filtered $G$-isocrystal $(\varphi,\mathcal{F})$
to the joint dynamic of the induced operators $\varphi$ and $\mathcal{F}$
on the extended Bruhat-Tits building of $G$. 

\subsection{\label{subsec:GrFil}~}

Let $K$ be a field, $\Gamma\neq0$ a subgroup of $\mathbb{R}$, $V$
a finite dimensional $K$-vector space. A $\Gamma$-graduation on
$V$ is a collection of $K$-subspaces $\mathcal{G}=(\mathcal{G}_{\gamma})_{\gamma\in\Gamma}$
such that $V=\oplus_{\gamma\in\Gamma}\mathcal{G}_{\gamma}$. A $\Gamma$-filtration
on $V$ is a collection of $K$-subspaces $\mathcal{F}=(\mathcal{F}^{\gamma})_{\gamma\in\Gamma}$
for which there exists a $\Gamma$-graduation $\mathcal{G}$ such
that $\mathcal{F}=\Fi(\mathcal{G})$, i.e.~$\mathcal{F}^{\gamma}=\oplus_{\eta\geq\gamma}\mathcal{G}_{\eta}$
for every $\gamma\in\Gamma$. We call any such $\mathcal{G}$ a splitting
of $\mathcal{F}$. If $W$ is a $K$-subspace of $V$, then $\mathcal{F}\vert W=(\mathcal{F}^{\gamma}\cap W)_{\gamma\in\Gamma}$
is again a $\Gamma$-graduation on $W$. The degree of $\mathcal{F}$
is given by
\[
\deg(\mathcal{F})=\sum_{\gamma}\dim_{K}\left(\Gr_{\mathcal{F}}^{\gamma}\right)\cdot\gamma
\]
where $\Gr_{\mathcal{F}}^{\gamma}=\mathcal{F}^{\gamma}/\mathcal{F}_{+}^{\gamma}$
with $\mathcal{F}_{+}^{\gamma}=\cup_{\eta>\gamma}\mathcal{F}^{\eta}$
for every $\gamma\in\Gamma$. The scalar product of two $\Gamma$-filtrations
$\mathcal{F}_{1}$ and $\mathcal{F}_{2}$ on $V$ is defined by the
analogous formula
\begin{eqnarray*}
\left\langle \mathcal{F}_{1},\mathcal{F}_{2}\right\rangle  & = & \sum_{\gamma_{1},\gamma_{2}}\dim_{K}\left(\frac{\mathcal{F}_{1}^{\gamma_{1}}\cap\mathcal{F}_{2}^{\gamma_{2}}}{\mathcal{F}_{1,+}^{\gamma_{1}}\cap\mathcal{F}_{2}^{\gamma_{2}}+\mathcal{F}_{1}^{\gamma_{1}}\cap\mathcal{F}_{2,+}^{\gamma_{2}}}\right)\cdot\gamma_{1}\gamma_{2}\\
 & = & \sum_{\gamma}\gamma\cdot\deg\Gr_{\mathcal{F}_{1}}^{\gamma}(\mathcal{F}_{2})=\sum_{\gamma}\gamma\cdot\deg\Gr_{\mathcal{F}_{2}}^{\gamma}(\mathcal{F}_{1})
\end{eqnarray*}
where $\Gr_{\mathcal{F}_{i}}^{\gamma}(\mathcal{F}_{j})$ is the $\Gamma$-filtration
induced by $\mathcal{F}_{j}$ on $\Gr_{\mathcal{F}_{i}}^{\gamma}$.
We denote by $\Gra^{\Gamma}(K)$ and $\Fil^{\Gamma}(K)$ the categories
of $\Gamma$-graded and $\Gamma$-filtered finite dimensional $K$-vector
spaces, and equip them with their usual structure of exact $K$-linear
$\otimes$-categories. The $\Fi$ and $\Gr$-constructions yield exact
$K$-linear $\otimes$-functors
\[
\Fi:\Gra^{\Gamma}(K)\leftrightarrow\Fil^{\Gamma}(K):\Gr.
\]
The formula $(\iota\mathcal{G})_{\gamma}=\mathcal{G}_{-\gamma}$ defines
an involutive exact $\otimes$-endofunctor of $\Gra^{\Gamma}(K)$.

\subsection{~\label{subsec:Iso}}

Let $k$ be a perfect field of characteristic $p>0$, $W(k)$ the
ring of Witt vectors, $K=W(k)[\frac{1}{p}]$ its fraction field, $\sigma$
the Frobenius of $k$, $W(k)$ or $K$. An isocrystal over $k$ is
a finite dimensional $K$-vector space $V$ equipped with a $\sigma$-linear
isomorphism $\varphi$. Since $k$ is perfect, there is a canonical
slope decomposition $(V,\varphi)=\oplus_{\lambda\in\mathbb{Q}}(V_{\lambda},\varphi\vert V_{\lambda})$.
The $\varphi$-stable $K$-subspace $V_{\lambda}$ is the union of
all finitely generated $W(k)$-submodules $L$ of $V$ such that $\varphi^{(h)}(L)=p^{d}L$,
where $\lambda=\frac{d}{h}$ with $(d,h)\in\mathbb{Z}\times\mathbb{N}$.
We denote by $\Iso(k)$ the category of isocrystals over $k$. It
is a $\mathbb{Q}_{p}$-linear (non-neutral) Tannakian category and
the above slope decomposition yields an exact $\mathbb{Q}_{p}$-linear
$\otimes$-functor 
\[
\nu:\Iso(k)\rightarrow\Gra^{\mathbb{Q}}(K).
\]

\subsection{~\label{subsec:IsoFil}}

A $\Gamma$-filtered isocrystal over $k$ is an isocrystal $(V,\varphi)$
together with a $\Gamma$-filtration $\mathcal{F}$ on $V$. The $K$-vector
space $V$ thus carries three filtrations: the \emph{Hodge }$\Gamma$-filtration
$\mathcal{F}_{H}=\mathcal{F}$ and the pair of opposed $\varphi$-stable
\emph{Newton }$\mathbb{Q}$-filtrations $\mathcal{F}_{N}=\Fil(\mathcal{G}_{N})$
and $\mathcal{F}_{N}^{\iota}=\Fil(\iota\mathcal{G}_{N})$ attached
to the slope decomposition $\mathcal{G}_{N}=\nu(V,\varphi)$, given
by
\[
\mathcal{F}_{N}^{\lambda}=\oplus_{\eta\geq\lambda}V_{\eta}\quad\mbox{and}\quad\mathcal{F}_{N}^{\iota\lambda}=\oplus_{\eta\geq\lambda}V_{-\eta}.
\]
For $\Gamma=\mathbb{Z}$, these objects are usually called unramified
filtered isocrystals over $k$. 
\begin{lem}
\label{lem:DefWeakAdmEquScProd}The following conditions are equivalent:

\begin{enumerate}
\item $(V,\varphi,\mathcal{F})$ is weakly admissible, i.e.: 

\begin{enumerate}
\item $\deg(\mathcal{F}_{H})=\deg(\mathcal{F}_{N})$ and 
\item $\deg(\mathcal{F}_{H}\vert W)\leq\deg(\mathcal{F}_{N}\vert W)$ for
every $\varphi$-stable $K$-subspace $W$ of $V$, 
\end{enumerate}
\item For every $\varphi$-stable $\Delta$-filtration $\Xi$ on $V$, $\left\langle \mathcal{F}_{H},\Xi\right\rangle \leq\left\langle \mathcal{F}_{N},\Xi\right\rangle $.
\item For every $\varphi$-stable $\Delta$-filtration $\Xi$ on $V$, $\left\langle \mathcal{F}_{H},\Xi\right\rangle +\left\langle \mathcal{F}_{N}^{\iota},\Xi\right\rangle \leq0$.
\end{enumerate}
In $(2)$ and $(3)$, $\Delta$ is any non-trivial subgroup of $\mathbb{R}$. 
\end{lem}
\begin{proof}
For a $\varphi$-stable $K$-subspace $W$ of $V$ and $a,b,\delta$
in $\Delta$ with $a\leq b$, set 
\[
\Xi_{W,a,b}^{\delta}=\begin{cases}
V & \mbox{if }\delta\leq a,\\
W & \mbox{if }a<\delta\leq b,\\
0 & \mbox{if }b<\delta.
\end{cases}
\]
This defines a $\varphi$-stable $\Delta$-filtration $\Xi_{W,a,b}=(\Xi_{W,a,b}^{\delta})_{\delta\in\Delta}$
on $V$ with 
\begin{eqnarray*}
\left\langle \mathcal{F}_{H},\Xi_{W,a,b}\right\rangle  & = & a\cdot\deg(\mathcal{F}_{H})+(b-a)\cdot\deg(\mathcal{F}_{H}\vert W)\\
\mbox{and}\quad\left\langle \mathcal{F}_{N},\Xi_{W,a,b}\right\rangle  & = & a\cdot\deg(\mathcal{F}_{N})+(b-a)\cdot\deg(\mathcal{F}_{N}\vert W)
\end{eqnarray*}
Now $(2)$ implies that $\left\langle \mathcal{F}_{H},\Xi_{W,a,b}\right\rangle \leq\left\langle \mathcal{F}_{N},\Xi_{W,a,b}\right\rangle $
for every choice of $W$ and $a\leq b$, from which $(1)$ easily
follows. Conversely, let $\Xi$ be any $\varphi$-stable $\Delta$-filtration
on $V$. Write $\{\delta\in\Delta:\Gr_{\Xi}^{\delta}\neq0\}=\{\delta_{1}<\cdots<\delta_{n}\}$.
Then $\Xi^{\delta_{i}}$ is a $\varphi$-stable $K$-subspace of $V$.
Put $d_{H}(i)=\deg(\mathcal{F}_{H}\vert\Xi^{\delta_{i}})$ and $d_{N}(i)=\deg(\mathcal{F}_{N}\vert\Xi^{\delta_{i}})$.
Then 
\[
\left\langle \mathcal{F}_{H},\Xi\right\rangle =\sum_{i=1}^{n}\delta_{i}\cdot\Delta_{H}(i)\quad\mbox{and}\quad\left\langle \mathcal{F}_{N},\Xi\right\rangle =\sum_{i=1}^{n}\delta_{i}\cdot\Delta_{N}(i)
\]
where $\Delta_{\star}(i)=d_{\star}(i)-d_{\star}(i+1)$ for $1\leq i<n$
and $\Delta_{\star}(n)=d_{\star}(n)$, so that also
\[
\left\langle \mathcal{F}_{H},\Xi\right\rangle =\sum_{i=1}^{n}\Delta_{i}\cdot d_{H}(i)\quad\mbox{and}\quad\left\langle \mathcal{F}_{N},\Xi\right\rangle =\sum_{i=1}^{n}\Delta_{i}\cdot d_{N}(i)
\]
where $\Delta_{i}=\delta_{i}-\delta_{i-1}$ for $1<i\leq n$ and $\Delta_{1}=\delta_{1}$.
Now $(1)$ implies that $d_{H}(i)\leq d_{N}(i)$ for all $i$ and
$d_{H}(1)=d_{N}(1)$ from which $(2)$ follows since $\Delta_{i}>0$
for $i>1$. Note that with these notations, we also find that 
\[
\left\langle \mathcal{F}_{N}^{\iota},\Xi\right\rangle =\sum_{i=1}^{n}\Delta_{i}\cdot d_{N}^{\iota}(i)\quad\mbox{with}\quad d_{N}^{\iota}(i)=\deg(\mathcal{F}_{N}^{\iota}\vert\Xi^{\delta_{i}}).
\]
But $\deg(\mathcal{F}_{N})+\deg(\mathcal{F}_{N}^{\iota})=0$ and more
generally $\deg(\mathcal{F}_{N}\vert W)+\deg(\mathcal{F}_{N}^{\iota}\vert W)=0$
for every $\varphi$-stable $K$-subspace $W$ of $V$ by functoriality
of the slope $\mathbb{Q}$-graduation, thus $d_{N}(i)+d_{N}^{\iota}(i)=0$
for all $i$. Therefore $\left\langle \mathcal{F}_{N},\Xi\right\rangle +\left\langle \mathcal{F}_{N}^{\iota},\Xi\right\rangle =0$
for every $\varphi$-stable $\mathbb{R}$-graduation $\Xi$ on $V$,
which proves that $(2)\Leftrightarrow(3)$. 
\end{proof}
\noindent We denote by $\Fil^{\Gamma}\Iso(k)$ the category of $\Gamma$-filtered
isocrystals over $k$ and denote by $\Fil^{\Gamma}\Iso(k)^{wa}$ its
full sub-category of weakly admissible objects, as defined in the
previous lemma. These are both exact $\mathbb{Q}_{p}$-linear $\otimes$-categories,
the smaller one is also abelian, and it is even a \emph{neutral }Tannakian
category when $\Gamma=\mathbb{Z}$ \cite{CoFo00}.

\subsection{~\label{subsec:FilIsoG}}

Fix a reductive group $G$ over $\mathbb{Q}_{p}$, let $\Rep(G)$
be the Tannakian category of algebraic representations of $G$ on
finite dimensional $\mathbb{Q}_{p}$-vector spaces, and let 
\[
V:\Rep(G)\rightarrow\Vect(\mathbb{Q}_{p})
\]
be the natural fiber functor. A $\Gamma$-graduation on $V_{K}=V\otimes K$
(resp.~$\Gamma$-filtration on $V_{K}$, $G$-isocrystal over $k$
or $\Gamma$-filtered $G$-isocrystal over $k$) is a factorization
of 
\[
V_{K}:\Rep(G)\rightarrow\Vect(K)
\]
through the natural fiber functor of the relevant category $\Gra^{\Gamma}(K)$
(resp.~$\Fil^{\Gamma}(K)$, $\Iso(k)$, $\Fil^{\Gamma}\Iso(k)$).
We require these factorizations to be exact and compatible with the
$\otimes$-products and their neutral objects. The set of all such
factorizations is the set of $K$-valued points of a smooth and separated
$\mathbb{Q}_{p}$-scheme $\mathbb{G}^{\Gamma}(G)$ (resp. $\mathbb{F}^{\Gamma}(G)$,
$G$, $G\times\mathbb{F}^{\Gamma}(G)$). The $G$-isocrystal attached
to $b\in G(K)$ maps $\rho\in\Rep(G)$ to $(V_{K}(\rho),\rho(b)\circ\mathrm{Id}\otimes\sigma)$
(this would be $(V,\rho)\mapsto(V_{K},\rho(b)\sigma)$ in standard
notations). We identify $b$ with the corresponding Frobenius element
$\varphi=(b,\sigma)$ in $G(K)\rtimes\left\langle \sigma\right\rangle $.
For $\rho\in\Rep(G)$, we denote by $\mathcal{G}(\rho)$, $\mathcal{F}(\rho)$
or $\varphi(\rho)$ the $\Gamma$-graduation, $\Gamma$-filtration
or Frobenius on $V_{K}(\rho)=V(\rho)\otimes K$ attached to a $\Gamma$-graduation
$\mathcal{G}$, $\Gamma$-filtration $\mathcal{F}$ or Frobenius $\varphi$
on $V_{K}$. We say that a $\Gamma$-filtered $G$-isocrystal 
\[
(\varphi,\mathcal{F}):\Rep(G)\rightarrow\Fil^{\Gamma}\Iso(k)
\]
 is weakly admissible if it factors through the subcategory $\Fil^{\Gamma}\Iso(k)^{wa}$
of $\Fil^{\Gamma}\Iso(k)$.
\begin{rem}
The $G$-isocrystals or $\Gamma$-filtered $G$-isocrystals considered
in this paper are trivial in the sense that the underlying fiber functor
is required to be the trivial fiber functor $V_{K}$. We caution our
reader not to apply our results carelessly on more general $G$-isocrystals
or $\Gamma$-filtered $G$-isocrystals, unless he or she has checked
that the underlying fiber functors are at least isomorphic (over $K$)
to the trivial one. 
\end{rem}

\subsection{~\label{subsec:FilIsoGRep}}

There is a $G$-equivariant sequence of morphisms of $\mathbb{Q}_{p}$-schemes
\[
\xymatrix{\mathbb{G}^{\Gamma}(G)\ar@{->>}[r]\sp(0.45){\Fi} & \mathbb{F}^{\Gamma}(G)\ar@{->>}[r]\sp(0.45){t} & \mathbb{C}^{\Gamma}(G)}
\]
where $\mathbb{C}^{\Gamma}(G)=G\backslash\mathbb{G}^{\Gamma}(G)=G\backslash\mathbb{F}^{\Gamma}(G)$
and the $\Fi$-morphism is induced by the previous $\Fi$-functors.
We set $\mathbf{G}^{\Gamma}(G_{K})=\mathbb{G}^{\Gamma}(G)(K)$, $\mathbf{F}^{\Gamma}(G_{K})=\mathbb{F}^{\Gamma}(G)(K)$
and denote by $\mathbf{C}^{\Gamma}(G_{K})$ the image of $\mathbf{F}^{\Gamma}(G_{K})$
in $\mathbb{C}^{\Gamma}(G)(K)$ under the type morphism $t$, giving
rise to a $G(K)$-equivariant sequence of surjective maps \cite[4.1]{Co14}
\[
\xymatrix{\mathbf{G}^{\Gamma}(G_{K})\ar@{->>}[r]\sp(0.47){\Fi} & \mathbf{F}^{\Gamma}(G_{K})\ar@{->>}[r]\sp(0.45){t} & \mathbf{C}^{\Gamma}(G_{K}).}
\]
Then $\mathbf{C}^{\Gamma}(G_{K})=G(K)\backslash\mathbf{G}^{\Gamma}(G_{K})=G(K)\backslash\mathbf{F}^{\Gamma}(G_{K})$,
and it is a monoid. When $\Gamma=\mathbb{R}$, it is the usual (relative)
Weyl cone attached to $G$ over $K$, and it comes equipped with a
partial order (the dominance order). If $S\subset G_{K}$ is a maximal
$K$-split torus and $B\subset G_{K}$ is a minimal parabolic subgroup
of $G_{K}$ containing the centralizer of $S$, then $\mathbf{C}^{\Gamma}(G_{K})\simeq\Hom^{+}(X^{\ast}(S),\Gamma)$,
where $\Hom^{+}(X^{\ast}(S),\Gamma)$ is the monoid of all morphisms
$f:X^{\ast}(S)\rightarrow\Gamma$ which are non-negative on the roots
of $S$ in the Lie algebra of $B$ \cite[4.1.10]{Co14}; when $\Gamma=\mathbb{R}$,
$f_{1}\leq f_{2}$ if and only if $f_{2}-f_{1}$ is a non-negative
linear combination of $B$-positive (relative) coroots \cite[2.4 and 5.1.2]{Co14}. 

The Newton slope decomposition of section~\ref{subsec:Iso} yields
a morphism
\[
\nu:G(K)\rightarrow\mathbf{G}^{\mathbb{Q}}(G_{K})
\]
such that $\nu(g\diamond b)=g\cdot\nu(b)$ for $g,b\in G(K)$ with
$g\diamond b=gb\sigma(g)^{-1}$, i.e.~a $G(K)$-equivariant morphism
$\varphi\mapsto\nu(\varphi)$ from the $G(K)$-coset $G(K)\cdot\sigma$
of Frobenius elements in $G(K)\rtimes\left\langle \sigma\right\rangle $
(with the action by conjugation) to $\mathbf{G}^{\mathbb{Q}}(G_{K})$. 

All of these constructions are covariantly functorial in $G$, $\Gamma$
and $k$.

\subsection{~\label{subsec:3FiltrationsG}}

A $\Gamma$-filtered $G$-isocrystal $(\varphi,\mathcal{F})\in G(K)\cdot\sigma\times\mathbf{F}^{\Gamma}(G_{K})$
thus again defines three filtrations on $V_{K}$: the \emph{Hodge
}$\Gamma$-filtration $\mathcal{F}_{H}=\mathcal{F}$ in $\mathbf{F}^{\Gamma}(G_{K})$
and the pair of opposed $\varphi$-stable \emph{Newton $\mathbb{Q}$-}filtrations
$\mathcal{F}_{N}=\Fi(\mathcal{G}_{N})$ and $\mathcal{F}_{N}^{\iota}=\Fi(\iota\mathcal{G}_{N})$
in $\mathbf{F}^{\mathbb{Q}}(G_{K})$ attached to the slope $\mathbb{Q}$-graduation
$\mathcal{G}_{N}=\nu(\varphi)$ in $\mathbf{G}^{\mathbb{Q}}(G_{K})$.

\subsection{~\label{subsec:Build}}

On the other hand, a $\Gamma$-filtered $G$-isocrystal $(\varphi,\mathcal{F})$
also ``acts'' on the extended Bruhat-Tits building $\mathbf{B}^{e}(G_{K})$
of $G$ over $K$, through the classical action of the group $G(K)\rtimes\left\langle \sigma\right\rangle $
on $\mathbf{B}^{e}(G_{K})$ and the canonical $G(K)\rtimes\left\langle \sigma\right\rangle $-equivariant
$+$-operation 
\[
+:\mathbf{B}^{e}(G_{K})\times\mathbf{F}^{\mathbb{R}}(G_{K})\rightarrow\mathbf{B}^{e}(G_{K})
\]
which is defined in~\cite[6.2.5]{Co14}. We recall that the extended
Bruhat-Tits building is the direct product of the usual (or reduced)
isogeny-invariant Bruhat-Tits building $\mathbf{B}(G_{K})$ by $X_{\ast}(A)\otimes\mathbb{R}$,
where $A$ is the maximal $K$-split torus in the center of $G_{K}$
and $X_{\ast}(A)$ is the group of cocaracters of $A$. The extended
Bruhat-Tits building $\mathbf{B}^{e}(G_{K})$ and vectorial Tits building
$\mathbf{F}^{\mathbb{R}}(G_{K})$ are covered by apartments, both
indexed by the maximal $K$-split subtori of $G_{K}$; for any such
torus, the corresponding apartment in the Bruhat-Tits building is
canonically equipped with the structure of an affine space whose underlying
vector space is the matching apartment in the vectorial Tits building,
and the above $+$-map is obtained by gluing the resulting faithfully
transitive actions on all apartments. To simplify our notations, we
set 
\[
X=\mathbf{B}^{e}(G_{K})=\mathbf{B}(G_{K})\times X_{\ast}(A)\otimes\mathbb{R}.
\]
For $\rho\in\Rep(G)$, $\varphi(\rho)$ and $\mathcal{F}(\rho)$ similarly
act on 
\[
X(\rho)=\mathbf{B}^{e}(GL(V(\rho))_{K}).
\]

\begin{defn}
For a $\Gamma$-filtered $G$-isocrystal $(\varphi,\mathcal{F})$,
we set 
\[
X_{\varphi}(\mathcal{F})=\left\{ x\in X:\varphi(x)+\mathcal{F}=x\right\} =\left\{ x\in X:\alpha(x)=x\right\} 
\]
where $\alpha(x)=\varphi(x)+\mathcal{F}$ for all $x$ in $X$. For
$\rho\in\Rep(G)$, we set 
\[
X_{\varphi}(\mathcal{F})(\rho)=\left\{ x\in X(\rho):\varphi(\rho)(x)+\mathcal{F}(\rho)=x\right\} =\left\{ x\in X(\rho):\alpha(\rho)(x)=x\right\} 
\]
where $\alpha(\rho)(x)=\varphi(\rho)(x)+\mathcal{F}(\rho)$ for all
$x$ in $X(\rho)$.
\end{defn}

\subsection{~\label{subsec:BuildDist}}

As explained in \cite[\S 5]{Co14}, the choice of a faithful $\tau\in\Rep(G)$
induces a bunch of numerical functions on our buildings. The $G(K)\rtimes\left\langle \sigma\right\rangle $-invariant
``scalar product'' 
\[
\left\langle -,-\right\rangle _{\tau}:\mathbf{F}^{\mathbb{R}}(G_{K})\times\mathbf{F}^{\mathbb{R}}(G_{K})\rightarrow\mathbb{R},\qquad\left\langle \mathcal{F}_{1},\mathcal{F}_{2}\right\rangle _{\tau}=\left\langle \mathcal{F}_{1}(\tau),\mathcal{F}_{2}(\tau)\right\rangle 
\]
yields a $\left\langle \sigma\right\rangle $-invariant length function
\[
\left\Vert -\right\Vert _{\tau}:\mathbf{C}^{\mathbb{R}}(G_{K})\rightarrow\mathbb{R}_{+},\qquad\left\Vert t(\mathcal{F})\right\Vert _{\tau}=\left\Vert \mathcal{F}\right\Vert _{\tau}=\left\langle \mathcal{F},\mathcal{F}\right\rangle _{\tau}^{\frac{1}{2}}
\]
whose composition with the $G(K)\rtimes\left\langle \sigma\right\rangle $-equivariant
``vector valued distance''
\[
\mathbf{d}:\mathbf{B}^{e}(G_{K})\times\mathbf{B}^{e}(G_{K})\rightarrow\mathbf{C}^{\mathbb{R}}(G_{K}),\qquad\mathbf{d}(x,x+\mathcal{F})=t(\mathcal{F})
\]
yields a genuine $G(K)\rtimes\left\langle \sigma\right\rangle $-invariant
distance
\[
d_{\tau}:\mathbf{B}^{e}(G_{K})\times\mathbf{B}^{e}(G_{K})\rightarrow\mathbb{R}_{+}
\]
which turns $\mathbf{B}^{e}(G_{K})$ into a complete CAT(0)-space.
The above $+$-map defines a $G(K)\rtimes\left\langle \sigma\right\rangle $-equivariant
``action'' of $\mathbf{F}^{\mathbb{R}}(G_{K})$ on $X=\mathbf{B}^{e}(G_{K})$
by non-expanding maps, which yields a $G(K)\rtimes\left\langle \sigma\right\rangle $-equivariant
identification of $\mathbf{F}^{\mathbb{R}}(G_{K})$ with the cone
$\mathscr{C}(\partial X)$ on the visual boundary $\partial X$ of
$X$ (acting on $X$ as in \cite{Co13}). Under this identification,
$\partial X=\{\xi\in\mathbf{F}^{\mathbb{R}}(G_{K}):\left\Vert \xi\right\Vert _{\tau}=1\}$.
On the other hand, the Frobenius $\varphi$ is an isometry of $(X,d_{\tau})$.
Thus $\alpha$ is a non-expanding map of $(X,d_{\tau})$. In particular,
\[
X_{\varphi}(\mathcal{F})\neq\emptyset\iff\alpha\mbox{ has bounded orbits}.
\]

\begin{lem}
The isometry $\varphi$ of $(X,d_{\tau})$ is semi-simple, i.e.~$\Min(\varphi)\neq\emptyset$
where 
\begin{eqnarray*}
\Min(\varphi) & = & \left\{ x\in X:d_{\tau}(x,\varphi(x))=\min(\varphi)\right\} \\
\mbox{with}\quad\mbox{\ensuremath{\min}}(\varphi) & = & \inf\left\{ d_{\tau}(x,\varphi(x)):x\in X\right\} .
\end{eqnarray*}
Moreover, $\varphi(x)=x+\mathcal{F}_{N}^{\iota}$ for every $x\in\Min(\varphi)$
and $\mbox{\ensuremath{\mathscr{C}}}(\partial\Min(\varphi))$ is the
fixed point set of $\varphi$ acting on $\mathscr{C}(\partial X)=\mathbf{F}^{\mathbb{R}}(G_{K})$,
i.e. the set of $\varphi$-stable $\mathbb{R}$-filtrations on $V_{K}$.
\end{lem}
\begin{proof}
The first two assertions are established in~\cite{CoNi16} when $k$
is algebraically closed and the general case follows from~\cite[II.6.2 (4)]{BrHa99}.
The final claim holds true for any semi-simple isometry $\varphi$
of a CAT(0)-space $(X,d)$. Indeed, $\partial\Min(\varphi)$ is contained
in $(\partial X)^{\varphi=\mathrm{Id}}$ by~\cite[II.6.8 (4)]{BrHa99}.
Conversely, suppose that $\xi\in\partial X$ is fixed by $\varphi$
and choose some $x$ in $\Min(\varphi)\neq\emptyset$. Then for every
$t\geq0$, 
\[
\varphi(x+t\xi)=\varphi(x)+t\varphi(\xi)=\varphi(x)+t\xi,
\]
thus $x+t\xi$ also belongs to $\Min(\varphi)$ since 
\[
\min(\varphi)\leq d(x+t\xi,\varphi(x+t\xi))=d(x+t\xi,\varphi(x)+t\xi)\leq d(x,\varphi(x))=\min(\varphi)
\]
by convexity of the CAT(0)-distance $d$, therefore $\xi\in\partial\Min(\varphi)$.
\end{proof}
\begin{rem}
We should rather write $\min_{\tau}(\varphi)$ to reflect the dependency
of $\min(\varphi)$ on the choice of $\tau$, but the subset $\Min(\varphi)$
of $X$ really does not depend upon that choice: it is precisely the
set of all $x\in X$ such that $\varphi(x)=x+\mathcal{F}_{N}^{\iota}$. 
\end{rem}

\subsection{~\label{subsec:BuildGL}}

For any $\rho\in\Rep(G)$, applying these constructions to $GL(V(\rho))$
and its tautological faithful representation on $V(\rho)$, we obtain
a canonical distance $d_{\rho}$ on $X(\rho)$ for which $\varphi(\rho)$
is an isometry while $\alpha(\rho)$ is a non-expanding map. Moreover:
\[
X_{\varphi}(\mathcal{F})(\rho)\neq\emptyset\iff\alpha(\rho)\mbox{\,\ has bounded orbits.}
\]
Here the cones $\mathscr{C}(\partial X(\rho))$ and $\mathscr{C}(\partial\Min(\varphi(\rho)))$
are respectively identified with the set of all $\mathbb{R}$-filtrations
on $V_{K}(\rho)$ and its subset of $\varphi(\rho)$-stable $\mathbb{R}$-filtrations.

\subsection{~\label{subsec:MainTheo1}}

Fix a faithful $\tau\in\Rep(G)$ and some subgroup $\Delta\neq0$
of $\mathbb{R}$. Then:
\begin{thm}
\label{thm:Main1}The following conditions are equivalent:

\begin{lyxlist}{XXX}
\item [{$(1)$}] $(\varphi,\mathcal{F})$ is weakly admissible.
\item [{$(1_{\rho})$}] For every $\rho\in\Rep(G)$, $(V_{K}(\rho),\varphi(\rho),\mathcal{F}(\rho))$
is weakly admissible.
\item [{$(1_{\tau})$}] $(V_{K}(\tau),\varphi(\tau),\mathcal{F}(\tau))$
is weakly admissible.
\item [{$(2^{\Delta})$}] For every $\varphi$-stable $\Delta$-filtration
$\Xi$ on $V_{K}$, 
\[
\left\langle \mathcal{F}_{H},\Xi\right\rangle _{\tau}\leq\left\langle \mathcal{F}_{N},\Xi\right\rangle _{\tau}.
\]
\item [{$(2_{\rho}^{\Delta})$}] For every $\rho\in\Rep(G)$ and every
$\varphi$-stable $\Delta$-filtration $\Xi$ on $V_{K}(\rho)$, 
\[
\left\langle \mathcal{F}_{H}(\rho),\Xi\right\rangle \leq\left\langle \mathcal{F}_{N}(\rho),\Xi\right\rangle .
\]
\item [{$(2_{\tau}^{\Delta})$}] For every $\varphi$-stable $\Delta$-filtration
$\Xi$ on $V_{K}(\tau)$, 
\[
\left\langle \mathcal{F}_{H}(\tau),\Xi\right\rangle \leq\left\langle \mathcal{F}_{N}(\tau),\Xi\right\rangle .
\]
\item [{$(3)$}] $X_{\varphi}(\mathcal{F})\neq\emptyset$.
\item [{$(3_{\rho})$}] For every $\rho\in\Rep(G)$, $X_{\varphi}(\mathcal{F})(\rho)\neq\emptyset$.
\item [{$(3_{\tau})$}] $X_{\varphi}(\mathcal{F})(\tau)\neq\emptyset$.
\item [{$(4)$}] $\alpha$ has bounded orbits in $(X,d_{\tau})$.
\item [{$(4_{\rho})$}] For every $\rho\in\Rep(G)$, $\alpha(\rho)$ has
bounded orbits in $(X(\rho),d_{\rho})$.
\item [{$(4_{\tau})$}] $\alpha(\tau)$ has bounded orbits in $(X(\tau),d_{\tau})$.
\end{lyxlist}
\end{thm}
\begin{rem}
It follows from the proof of Lemma~\ref{lem:DefWeakAdmEquScProd}
that here also
\[
\left\langle \mathcal{F}_{N},\Xi\right\rangle _{\tau}+\left\langle \mathcal{F}_{N}^{\iota},\Xi\right\rangle _{\tau}=0
\]
for every $\varphi$-stable $\mathbb{R}$-filtration $\Xi$ on $V_{K}$. 
\end{rem}
\begin{proof}
Obviously $(1)\Leftrightarrow(1_{\rho})$, $(2_{\tau}^{\Delta})\Rightarrow(2^{\Delta})$
and $(x_{\rho})\Rightarrow(x_{\tau})$ for $x\in\{1,2,3,4\}$. We
have already seen that $(3_{x})\Leftrightarrow(4_{x})$ for all $x\in\{\emptyset,\rho,\tau\}$
and Lemma~\ref{lem:DefWeakAdmEquScProd} shows that $(1_{\rho})\Leftrightarrow(2_{\rho}^{\Delta})$
and $(1_{\tau})\Leftrightarrow(2_{\tau}^{\Delta})$. Set $(2_{x})=(2_{x}^{\mathbb{R}})$
for $x\in\{\emptyset,\rho,\tau\}$ and let $(2')$, $(3')$ and $(4')$
be $(2)$, $(3)$ and $(4)$ for the base change $(\varphi',\mathcal{F}')$
of $(\varphi,\mathcal{F})$ to an algebraic closure $k'$ of $k$.
Then obviously $(2_{x})\Rightarrow(2_{x}^{\Delta}),$ $(2')\Rightarrow(2)$
and $(4)\Leftrightarrow(4')$ \textendash{} thus also $(3')\Leftrightarrow(3)$.
The main theorem of \cite{Co13} asserts that $(2)\Leftrightarrow(3)$
provided that $\varphi$ satisfies a certain decency condition, which
holds true by the main result of~\cite{CoNi16} when $k$ is algebraically
closed. Therefore $(2')\Leftrightarrow(3')$. We will show below that
$(2^{\Delta})\Rightarrow(2)$ (\ref{subsec:2Dimplies2}), $(2)\Rightarrow(2')$
(\ref{subsec:2implies2'}) and $(4)\Rightarrow(4_{\rho})$ (\ref{subsec:4implies4tau}).
Since $(3_{\rho})\Leftrightarrow(2_{\rho})$ and $(3_{\tau})\Leftrightarrow(2_{\tau})$
are special cases of $(3)\Leftrightarrow(3')\Leftrightarrow(2')\Leftrightarrow(2)$
(applied respectively to $G=GL(V(\rho))$ and $G=GL(V(\tau)))$, the
theorem follows. 
\end{proof}
\begin{rem}
Apart from the added generality of allowing $\Gamma\neq\mathbb{Z}$,
the equivalence of most of these conditions is either well-known or
trivial and can be found in various places. For instance $(1)\Leftrightarrow(1_{\tau})$
follows from~\cite[1.18]{RaZi96} and $(1_{\tau})\Leftrightarrow(3_{\tau})$
is Laffaille's criterion for weak admissibility \cite[3.2]{Laf80}.
Our only new contribution is $(2)\Leftrightarrow(3)$, which is essentially
our metric reformulation of Laffaille's algebraic proof in \cite{Co13}. 
\end{rem}

\subsection{\label{subsec:2Dimplies2}$(2^{\Delta})\Rightarrow(2)$}

Suppose that $(2)$ does not hold: there exists a $\varphi$-stable
$\mathbb{R}$-filtration $\Xi$ on $V_{K}$ such that $\left\langle \mathcal{F}_{H},\Xi\right\rangle _{\tau}>\left\langle \mathcal{F}_{N},\Xi\right\rangle _{\tau}$.
By~\cite[4.1.13]{Co14}, there exists a maximal split torus $S$
of $G_{K}$ with character group $M=\Hom_{K}(S,\mathbb{G}_{m,K})$
and a morphism $f:M\rightarrow\mathbb{R}$ such that for every $\rho\in\Rep(G)$
and $\gamma\in\mathbb{R}$, 
\[
\Xi(\rho)^{\gamma}=\oplus_{m\in M,\,f(m)\geq\gamma}V(\rho)_{m}\quad\mbox{in}\quad V_{K}(\rho)=V(\rho_{K})
\]
where $V(\rho_{K})=\oplus_{m\in M}V(\rho)_{m}$ is the weight decomposition
of the restriction of $\rho_{K}$ to $S$. The image of $f$ is a
finitely generated subgroup $Q$ of $\mathbb{R}$. Since $\Delta$
is a non-trivial subgroup of $\mathbb{R}$, there exists a sequence
of morphisms $q_{n}:Q\rightarrow\Delta$ such that $\frac{1}{n}q_{n}:Q\rightarrow\mathbb{R}$
converges simply to the given embedding $Q\hookrightarrow\mathbb{R}$.
The formula 
\[
\Xi_{n}(\rho)^{\gamma}=\oplus_{m\in M,\,q_{n}\circ f(m)\geq\gamma}V(\rho)_{m}
\]
then defines a $\Delta$-filtration $\Xi_{n}\in\mathbf{F}^{\Delta}(G_{K})$
with $\frac{1}{n}\Xi_{n}\rightarrow\Xi$ in $(\mathbf{F}^{\mathbb{R}}(G_{K}),d_{\tau})$,
thus also 
\[
\left\langle \mathcal{F}_{H},\Xi_{n}\right\rangle _{\tau}=n\left\langle \mathcal{F}_{H},{\textstyle \frac{1}{n}}\Xi_{n}\right\rangle _{\tau}>n\left\langle \mathcal{F}_{N},{\textstyle \frac{1}{n}}\Xi_{n}\right\rangle _{\tau}=\left\langle \mathcal{F}_{N},\Xi_{n}\right\rangle _{\tau}
\]
for $n\gg0$. For any $\tau$ in $\Rep(G)$, $(m_{1},m_{2})$ in $\{m\in M:V(\rho)_{m}\neq0\}$
and $n\gg0$, 
\[
f(m_{1})-f(m_{2})=f(m_{1}-m_{2})\quad\mbox{and}\quad q_{n}\circ f(m_{1})-q_{n}\circ f(m_{2})=q_{n}\circ f(m_{1}-m_{2})
\]
have the same sign in $\{0,\pm1\}$, thus for every $\rho$ in $\Rep(G)$
and $n\gg0$, 
\[
\{\Xi_{n}(\rho)^{\gamma}:\gamma\in\Delta\}=\{\Xi(\rho)^{\gamma}:\gamma\in\mathbb{R}\}.
\]
In particular, $\Xi_{n}(\tau)$ is fixed by $\varphi(\tau)$ for $n\gg0$.
But then $\Xi_{n}$ is a $\varphi$-stable $\Delta$-filtration on
$V_{K}$ by~\cite[4.2.10 or 3.11.12]{Co14}, thus $(2^{\Delta})$
also does not hold.

\subsection{\label{subsec:2implies2'}$(2)\Rightarrow(2')$}

Suppose that $(2')$ does not hold. Then by~\cite[Theorem 1, (4)]{Co13},
there is a unique $\xi$ in $\partial\Min(\varphi')$ such that $\left\langle \mathcal{F}_{H},\xi\right\rangle _{\tau}+\left\langle \mathcal{F}_{N}^{\iota},\xi\right\rangle _{\tau}>0$
is maximal. This uniqueness implies that $\xi\in\mathbf{F}^{\mathbb{R}}(G_{K'})$
is fixed by the group $\Gal(k'/k)$ of continuous automorphisms of
$K'/K$. Therefore $\xi\in\mathbf{F}^{\mathbb{R}}(G_{K})$ and $(2)$
also does not hold.

\subsection{\label{subsec:4implies4tau}$(4)\Rightarrow(4_{\rho})$}

Suppose that $(4)$ holds and fix $\rho\in\Rep(G)$ corresponding
to a morphism $f:G\rightarrow H=GL(V(\rho))$. By~\cite[5.8.4]{Co14},
there is a finite Galois extension $L$ of $\mathbb{Q}_{p}$ and for
any factor $\tilde{K}$ of $L\otimes K$, a map $f:\mathbf{B}^{e}(G_{\tilde{K}})\rightarrow\mathbf{B}^{e}(H_{\tilde{K}})$
such that 
\[
f(g\cdot x)=f(g)\cdot f(x),\quad f(x+\mathcal{G})=f(x)+f(\mathcal{G})\quad\mbox{and}\quad f(\theta\cdot x)=\theta\cdot f(x)
\]
for every $x\in\mathbf{B}^{e}(G_{\tilde{K}})$, $g\in G(\tilde{K})$,
$\mathcal{G}\in\mathbf{F}^{\mathbb{R}}(G_{\tilde{K}})$ and $\theta\in\Aut(\tilde{K}/\mathbb{Q}_{p})$.
Fix an extension $\tilde{\sigma}\in\Aut(\tilde{K}/\mathbb{Q}_{p})$
of $\sigma\in\Aut(K/\mathbb{Q}_{p})$, let $\tilde{\varphi}=(b,\tilde{\sigma})$
be the corresponding extension of $\varphi=(b,\sigma)$, write $\tilde{\alpha}$
and $\tilde{\alpha}(\rho)$ for the induced extensions of $\alpha$
and $\alpha(\rho)$ to $\mathbf{B}^{e}(G_{\tilde{K}})\supset\mathbf{B}^{e}(G_{K})$
and $\mathbf{B}^{e}(H_{\tilde{K}})\supset\mathbf{B}^{e}(H_{K})$.
Then $f(\tilde{\alpha}(x))=\tilde{\alpha}(\rho)(f(x))$ for every
$x\in\mathbf{B}^{e}(G_{\tilde{K}})$. Since $\alpha$ has a fixed
point $x$ in $\mathbf{B}^{e}(G_{K})$, $f(x)$ is a fixed point of
$\tilde{\alpha}(\rho)$ in $\mathbf{B}^{e}(H_{\tilde{K}})$, thus
$\tilde{\alpha}(\rho)$ has bounded orbits, and so does $\alpha(\rho)=\tilde{\alpha}(\rho)\vert\mathbf{B}^{e}(H_{K})$. 

\section{Mazur's Inequality\label{sec:Mazur}}

In this section, we explain our variant of Mazur's inequality (\ref{subsec:MazursIneq1})
and prove its converse (Theorem~\ref{thm:Main2MazurRec}, $\nu\leq\mu^{\sharp}\Longrightarrow X_{\varphi}(\mu)\neq\emptyset$
), following the Fontaine-Rapoport strategy, i.e.~using a covering
of our affine Deligne-Lusztig sets (the $X_{\varphi}(\mu)$'s) by
our subsets of strongly divisible points (the $X_{\varphi}(\mathcal{F})$'s,
see~ \ref{subsec:DefAffDelSets} and \ref{subsec:AnalysisOfAddDel=000026StronglyDiv}),
whose non-emptiness is characterized by our variant of Laffaille's
theorem (Theorem~\ref{thm:Main1}). 

\subsection{~\label{subsec:DefAffDelSets}}

Fix a Frobenius $\varphi=b\cdot\sigma\in G(K)\rtimes\left\langle \sigma\right\rangle $.
For $\mu\in\mathbf{C}^{\mathbb{R}}(G_{K})$, we define 
\[
X_{\varphi}(\mu)=\left\{ x\in X:\,\mathbf{d}(\varphi(x),x)=\mu\right\} .
\]
For any $x\in X_{\varphi}(\mu)$, there is an $\mathcal{F}\in\mathbf{F}^{\mathbb{R}}(G_{K})$
of type $\mu$ such that $x=\varphi(x)+\mathcal{F}$, i.e.~$x\in X_{\varphi}(\mathcal{F})$.
Conversely, $X_{\varphi}(\mathcal{F})\subset X_{\varphi}(\mu)$ for
any $\mathcal{F}$ of type $\mu$, thus 
\[
X_{\varphi}(\mu)=\cup_{\mathcal{F}\in\mathbf{F}^{\mathbb{R}}(G_{K}):t(\mathcal{F})=\mu}X_{\varphi}(\mathcal{F})=\cup_{\mathcal{F}\in\mathbf{Adm}_{\varphi}(\mu)}X_{\varphi}(\mathcal{F})
\]
and $X_{\varphi}(\mu)\neq\emptyset\Leftrightarrow\mathbf{Adm}_{\varphi}(\mu)\neq\emptyset$
by the previous theorem, where 
\[
\mathbf{Adm}_{\varphi}(\mu)=\left\{ \mathcal{F}\in\mathbf{F}^{\mathbb{R}}(G_{K}):\,(\varphi,\mathcal{F})\mbox{ is weakly admissible and }t(\mathcal{F})=\mu\right\} .
\]

\subsection{~\label{subsec:MazursIneq1}}

Write $\nu=t(\mathcal{F}_{N})$. As $\mathcal{F}_{N}\in\mathbf{F}^{\mathbb{Q}}(G_{K})$
is fixed by $\varphi=b\cdot\sigma$ and the type map $t:\mathbf{F}^{\mathbb{Q}}(G_{K})\rightarrow\mathbf{C}^{\mathbb{Q}}(G_{K})$
is $G(K)$-invariant, $\nu$ belongs to the fixed point set of $\sigma$
in $\mathbf{C}^{\mathbb{Q}}(G_{K})$. Since $\mathbf{C}^{\mathbb{Q}}(G_{K})$
is a uniquely divisible commutative monoid on which $\left\langle \sigma\right\rangle $
acts with finite orbits, averaging over these orbits yields a retraction
$\mu\mapsto\mu^{\sharp}$ onto the fixed point set of $\sigma$. When
$k$ is algebraically closed, \cite[Théorème 6]{CoNi16} shows that
\[
X_{\varphi}(\mu)\neq\emptyset\Longrightarrow\nu\leq\mu^{\sharp}\quad\mbox{in }\left(\mathbf{C}^{\mathbb{Q}}(G_{K}),\leq\right)
\]
where $\leq$ is the dominance partial order on $\mathbf{C}^{\mathbb{Q}}(G_{K})$.
This is still true over any perfect field $k$, because if $k'$ is
an algebraic closure of $k$ and $K'=\Frac(W(k'))$, then $\mathbf{C}^{\mathbb{Q}}(G_{K})\hookrightarrow\mathbf{C}^{\mathbb{Q}}(G_{K'})$
is a $\left\langle \sigma\right\rangle $-equivariant embedding of
partially ordered commutative monoids. We will establish below that
conversely, $\nu\leq\mu^{\sharp}$ implies $X_{\varphi}(\mu)\neq\emptyset$
when $k$ is algebraically closed and $G$ is unramified over $\mathbb{Q}_{p}$.
The latter assumption seems superfluous, but some variant of the former
is really needed: if $k=\mathbb{F}_{p}$ and $\varphi=\sigma$ (i.e.~$b=1$),
then $\nu=0$ but $X_{\varphi}(\mu)\neq\emptyset\iff\mu=0$. 

\subsection{~\label{subsec:AnalysisOfAddDel=000026StronglyDiv}}

The subset $X_{\varphi}(\mu)$ of $X=\mathbf{B}^{e}(G_{K})$ is closed
for the canonical topology of $X$, but it is typically not convex,
also $X_{\varphi}(\nu)=\Min(\varphi)$ is convex by~\cite[II.6.2]{BrHa99}.
On the other hand, $X_{\varphi}(\mathcal{F})$ is always closed and
convex by \cite[Theorem 1.3]{ChPh06}. The group 
\[
J=\Aut^{\otimes}(V_{K},\varphi)=\left\{ g\in G(K):\,\varphi g=g\varphi\right\} =\{g\in G(K):g\diamond b=b\}
\]
acts on $X_{\varphi}(\mu)$ and $\mathbf{Adm}_{\varphi}(\mu)$, and
$j\in J$ maps $X_{\varphi}(\mathcal{F})$ to $X_{\varphi}(j\cdot\mathcal{F})$.
The convex projection $p:X\twoheadrightarrow\Min(\varphi)$ is non-expanding,
$J$-equivariant, and of formation compatible with algebraic base
change on $k$. Indeed if $k'$ is a Galois extension of $k$ and
$p':X'\twoheadrightarrow\Min(\varphi')$ is the corresponding projection,
then $p'$ is also $\Gal(k'/k)$-equivariant. It thus maps $X=X'^{\Gal(k'/k)}$
\cite[2.6.1]{Ti79} to $\Min(\varphi')\cap X=\Min(\varphi)$ \cite[II.6.2]{BrHa99},
from which easily follows that $p'\vert X$ equals $p$. The projection
$p$ restricts to a quasi-isometric $J$-equivariant embedding $p:X_{\varphi}(\mu)\rightarrow\Min(\varphi)$,
which is even a $J$-equivariant quasi-isometry (as defined in~\cite[I.8.14]{BrHa99})
when $k$ is algebraically closed and $X_{\varphi}(\mu)\neq\emptyset$.
Indeed, the restriction of $p$ to $X_{\varphi}(\mu)$ has bounded
fibers and $J$ acts cocompactly on $\Min(\varphi)$ when $k$ is
algebraically closed by \cite{CoNi16}. In this case, the boundary
of $X_{\varphi}(\mu)$ equals that of $\Min(\varphi)$, i.e.~it is
the fixed point set of $\varphi$ in the boundary of $X$. The boundary
of $X_{\varphi}(\mathcal{F})$ is described in \cite[Theorem 1]{Co13}.

\subsection{~}

We now establish our second main result:
\begin{thm}
\label{thm:Main2MazurRec}If $G$ is unramified over $\mathbb{Q}_{p}$
and $k$ is algebraically closed, then
\[
X_{\varphi}(\mu)\neq\emptyset\iff\mathbf{Adm}_{\varphi}(\mu)\neq\emptyset\iff\nu\leq\mu^{\sharp}.
\]
\end{thm}
\begin{proof}
It only remains to show that $\nu\leq\mu^{\sharp}$ implies $\mathbf{Adm}_{\varphi}(\mu)\neq\emptyset$.
By~\cite[5.1.2]{Co14},
\[
\nu\leq\mu^{\sharp}\iff\forall z\in\mathbf{C}^{\mathbb{R}}(G_{K}):\,\left\langle \mu^{\sharp},z\right\rangle _{\tau}^{tr}\leq\left\langle \nu,z\right\rangle _{\tau}^{tr}
\]
where $\left\langle -,-\right\rangle _{\tau}^{tr}:\mathbf{C}^{\mathbb{R}}(G_{K})\times\mathbf{C}^{\mathbb{R}}(G_{K})\rightarrow\mathbb{R}$
is defined in~\cite[4.2.5]{Co14} and given by
\begin{eqnarray*}
\left\langle \mu_{1},\mu_{2}\right\rangle _{\tau}^{tr} & = & \inf\left\{ \left\langle \mathcal{F}_{1},\mathcal{F}_{2}\right\rangle _{\tau}:(\mathcal{F}_{1},\mathcal{F}_{2})\in\mathbf{F}^{\mathbb{R}}(G_{K})^{2},\,t(\mathcal{F}_{1},\mathcal{F}_{2})=(\mu_{1},\mu_{2})\right\} .
\end{eqnarray*}
This ``scalar product'' $\left\langle -,-\right\rangle _{\tau}^{tr}$
is $\sigma$-invariant and additive in both variables. It follows
that $\left\langle \mu^{\sharp},z\right\rangle _{\tau}^{tr}=\left\langle \mu,z\right\rangle _{\tau}^{tr}$
for every $z\in\mathbf{C}^{\mathbb{R}}(G_{K})$ fixed by $\sigma$,
so that 
\[
\nu\leq\mu^{\sharp}\Longrightarrow\forall z\in\mathbf{C}^{\mathbb{R}}(G_{K})^{\sigma=\mathrm{Id}}:\,\left\langle \mu,z\right\rangle _{\tau}^{tr}\leq\left\langle \nu,z\right\rangle _{\tau}^{tr}.
\]
Suppose therefore that $\nu\leq\mu^{\sharp}$. Then for any $\mathcal{F}_{H}\in\mathbf{F}^{\mathbb{R}}(G_{K})$
of type $t(\mathcal{F}_{H})=\mu$ and every $\varphi$-stable $\Xi\in\mathbf{F}^{\mathbb{R}}(G_{K})$,
since $t(\Xi)=z$ is fixed by $\sigma$ and $t(\mathcal{F}_{N})=\nu$,
\[
\left\langle t(\mathcal{F}_{H}),t(\Xi)\right\rangle _{\tau}^{tr}\leq\left\langle t(\mathcal{F}_{N}),t(\Xi)\right\rangle _{\tau}^{tr}\leq\left\langle \mathcal{F}_{N},\Xi\right\rangle _{\tau}.
\]
Thus $\mathcal{F}_{H}$ will belong to $\mathbf{Adm}_{\varphi}(\mu)$
if it happens to be in generic (= transverse) relative position \cite[4.2.5]{Co14}
with respect to every $\varphi$-stable $\Xi\in\mathbf{F}^{\mathbb{R}}(G_{K})$,
for then
\[
\left\langle \mathcal{F}_{H},\Xi\right\rangle _{\tau}=\left\langle t(\mathcal{F}_{H}),t(\Xi)\right\rangle _{\tau}^{tr}\leq\left\langle \mathcal{F}_{N},\Xi\right\rangle _{\tau}.
\]
It is therefore sufficient to show that the set of $\varphi$-stable
$\mathbb{R}$-filtrations on $V_{K}$ is a thin subset of $\mathbf{F}^{\mathbb{R}}(G_{K})$
in the sense of \cite[4.1.19]{Co14}. Since $k$ is algebraically
closed, we may assume that $\varphi$ satisfies a decency equation
as in~\cite[2.8]{CoNi16} for some integer $s>0$, in which case
every $\varphi$-stable filtration is also fixed by $\sigma^{s}$,
i.e.~defined over the fixed field $\mathbb{Q}_{p^{s}}$ of $\sigma^{s}$
in $K$, a finite (unramified) extension of $\mathbb{Q}_{p}$. Now
since $G$ is unramified over $\mathbb{Q}_{p}$, it admits a reductive
model over $\mathbb{Z}_{p}$, which we also denote by $G$. We may
now apply the thinness criterion of~\cite[4.1.19]{Co14}, thus obtaining
that the whole of $\mathbf{F}^{\mathbb{R}}(G_{\mathbb{Q}_{p^{s}}})$
is a thin subset of $\mathbf{F}^{\mathbb{R}}(G_{K})$, which proves
the theorem.
\end{proof}

\section{Filtered Isocrystals\label{sec:EquivForWeaklyAdm}}

\subsection{~}

The hardest implication in Theorem~\ref{thm:Main1} seems to be $(2^{\Delta})\Rightarrow(1)$:
\[
(2^{\Delta})\Rightarrow(2)\Rightarrow(2')\Rightarrow(4')\Rightarrow(4'_{\tau})\Rightarrow(3'_{\tau})\Rightarrow(2'_{\tau})\Rightarrow(2_{\tau})\Rightarrow(1_{\tau})\Rightarrow(1).
\]
The key steps are: the fixed point theorem of \cite{Co13}, used in
$(2')\Rightarrow(4')$, and some functoriality of Bruhat-Tits buildings,
used in $(4')\Rightarrow(4'_{\tau})$. However, the theory of Harder-Narasimhan
filtrations provides an easier proof of a more general result.

\subsection{~}

Fix an extension $L$ of $K=W(k)[\frac{1}{p}]$ and denote by $\Fil_{L}^{\Gamma}\Iso(k)$
the category of triples $(V,\varphi,\mathcal{F})$ where $(V,\varphi)$
is an isocrystal over $k$ and $\mathcal{F}=\mathcal{F}_{H}$ is a
$\Gamma$-filtration on $V_{L}=V\otimes_{K}L$. It is an exact $\otimes$-category
equipped with Harder-Narasimhan filtrations attached to the slope
function which is defined (for $V\neq0$) by
\[
\mu(V,\varphi,\mathcal{F})=\frac{\deg(\mathcal{F}_{H})-\deg(\mathcal{F}_{N})}{\dim_{K}V}\in\mathbb{Q}[\Gamma]\subset\mathbb{R}.
\]
Here $\mathbb{Q}[\Gamma]$ is the $\mathbb{Q}$-subspace of $\mathbb{R}$
spanned by $\Gamma$. We now have four filtrations at our disposal:
the \emph{Hodge} $\Gamma$-filtration $\mathcal{F}=\mathcal{F}_{H}$
on $V_{L}$, the $\varphi$-stable (opposed) \emph{Newton} $\mathbb{Q}$-filtrations
$\mathcal{F}_{N}$ and $\mathcal{F}_{N}^{\iota}$ on $V=V_{K}$, and
the $\varphi$-stable \emph{Harder-Narasimhan} $\mathbb{Q}[\Gamma]$-filtration
$\mathcal{F}_{HN}=\mathcal{F}_{HN}(V,\varphi,\mathcal{F})$ on $V=V_{K}$. 
\begin{lem}
\label{lem:DefWeakAdmEquScProdv2}The following conditions are equivalent:

\begin{enumerate}
\item $(V,\varphi,\mathcal{F})$ is weakly admissible, i.e.: 

\begin{enumerate}
\item $\deg(\mathcal{F}_{H})=\deg(\mathcal{F}_{N})$ and 
\item $\deg(\mathcal{F}_{H}\vert W)\leq\deg(\mathcal{F}_{N}\vert W)$ for
every $\varphi$-stable $K$-subspace $W$ of $V$, 
\end{enumerate}
\item For every $\varphi$-stable $\Delta$-filtration $\Xi$ on $V$, $\left\langle \mathcal{F}_{H},\Xi\right\rangle \leq\left\langle \mathcal{F}_{N},\Xi\right\rangle $.
\item For every $\varphi$-stable $\Delta$-filtration $\Xi$ on $V$, $\left\langle \mathcal{F}_{H},\Xi\right\rangle +\left\langle \mathcal{F}_{N}^{\iota},\Xi\right\rangle \leq0$.
\item $\left\langle \mathcal{F}_{H},\mathcal{F}_{HN}\right\rangle \leq\left\langle \mathcal{F}_{N},\mathcal{F}_{HN}\right\rangle $.
\item $\left\langle \mathcal{F}_{H},\mathcal{F}_{HN}\right\rangle +\left\langle \mathcal{F}_{N}^{\iota},\mathcal{F}_{HN}\right\rangle \leq0$.
\item $\mathcal{F}_{HN}=0$.
\end{enumerate}
In $(2)$ and $(3)$, $\Delta$ is any non-trivial subgroup of $\mathbb{R}$. 
\end{lem}
\begin{proof}
One proves $(1)\Leftrightarrow(2)\Leftrightarrow(3)$ and $(4)\Leftrightarrow(5)$
as in Lemma~\ref{lem:DefWeakAdmEquScProd}, moreover $(1)\Leftrightarrow(6)$
by definition of $\mathcal{F}_{HN}$ and obviously $(6)\Rightarrow(4),(5)$.
Finally $(4)\Rightarrow(6)$ because:
\begin{eqnarray*}
\left\langle \mathcal{F}_{H},\mathcal{F}_{HN}\right\rangle -\left\langle \mathcal{F}_{N},\mathcal{F}_{HN}\right\rangle  & = & \sum_{\gamma}\gamma\cdot\left(\deg\Gr_{\mathcal{F}_{HN}}^{\gamma}(\mathcal{F}_{H})-\deg\Gr_{\mathcal{F}_{HN}}^{\gamma}(\mathcal{F}_{N})\right)\\
 & = & \sum_{\gamma}\gamma\cdot\dim_{K}\left(\Gr_{\mathcal{F}_{HN}}^{\gamma}\right)\cdot\mu\left(\Gr_{\mathcal{F}_{HN}}^{\gamma}\right)\\
 & = & \sum_{\gamma}\gamma^{2}\cdot\dim_{K}\left(\Gr_{\mathcal{F}_{HN}}^{\gamma}\right)\\
 & = & \left\langle \mathcal{F}_{HN},\mathcal{F}_{HN}\right\rangle 
\end{eqnarray*}
using the definition of $\mathcal{F}_{HN}$ for the third equality.
\end{proof}

\subsection{~}

For a reductive group $G$ over $\mathbb{Q}_{p}$, a $\Gamma$-filtered
$G$-isocrystal over $L$ is an exact $\otimes$-functor factorization
$(\varphi,\mathcal{F})$ of the fiber functor $V_{K}:\Rep(G)\rightarrow\Vect(K)$
trough the natural fiber functor $\Fil_{L}^{\Gamma}\Iso(k)\rightarrow\Vect(K)$.
It again induces four filtrations on the relevant fiber functors:
the \emph{Hodge} $\Gamma$-filtration $\mathcal{F}_{H}=\mathcal{F}$
on $V_{L}$, the $\varphi$-stable (opposed) \emph{Newton} $\mathbb{Q}$-filtrations
$\mathcal{F}_{N}$ and $\mathcal{F}_{N}^{\iota}$ on $V_{K}$, and
the $\varphi$-stable \emph{Harder-Narasimhan} $\mathbb{Q}[\Gamma]$-filtration
$\mathcal{F}_{HN}=\mathcal{F}_{HN}(\varphi,\mathcal{F})$ on $V_{K}$
defined by 
\[
\forall\tau\in\Rep(G):\qquad\mathcal{F}_{HN}(\tau)=\mathcal{F}_{HN}(V_{K}(\tau),\varphi(\tau),\mathcal{F}(\tau)).
\]
It is not at all obvious that this formula indeed defines a filtration
on $V_{K}$: it does yield a factorization $\mathcal{F}_{HN}:\Rep(G)\rightarrow\Fil^{\mathbb{Q}[\Gamma]}(K)$
of $V_{K}$, but we have to check that the latter is exact and compatible
with the $\otimes$-products and their neutral objects. Since every
exact sequence in $\Rep(G)$ is split (we are in characteristic $0$),
exactness here amounts to additivity, which is obvious. The issue
is here the compatibility of Harder-Narasimhan filtrations with $\otimes$-products,
where we need to use the main result of \cite{To96}. We say that
$(\varphi,\mathcal{F}):\Rep(G)\rightarrow\Fil_{L}^{\Gamma}\Iso(k)$
is weakly admissible if it factors through the full subcategory $\Fil_{L}^{\Gamma}\Iso(k)^{wa}$
of weakly admissible objects. 

\subsection{~}

Fix a faithful $\tau\in\Rep(G)$ and some subgroup $\Delta\neq0$
of $\mathbb{R}$. Then:
\begin{thm}
\label{thm:Main1withL}The following conditions on $(\varphi,\mathcal{F})$
are equivalent:

\begin{lyxlist}{XXX}
\item [{$(1)$}] $(\varphi,\mathcal{F})$ is weakly admissible.
\item [{$(1_{\rho})$}] For every $\rho\in\Rep(G)$, $(V_{K}(\rho),\varphi(\rho),\mathcal{F}(\rho))$
is weakly admissible.
\item [{$(1_{\tau})$}] $(V_{K}(\tau),\varphi(\tau),\mathcal{F}(\tau))$
is weakly admissible.
\item [{$(2^{\Delta})$}] For every $\varphi$-stable $\Delta$-filtration
$\Xi$ on $V_{K}$, 
\[
\left\langle \mathcal{F}_{H},\Xi\right\rangle _{\tau}\leq\left\langle \mathcal{F}_{N},\Xi\right\rangle _{\tau}.
\]
\item [{$(2_{\rho}^{\Delta})$}] For every $\rho\in\Rep(G)$ and every
$\varphi$-stable $\Delta$-filtration $\Xi$ on $V_{K}(\rho)$, 
\[
\left\langle \mathcal{F}_{H}(\rho),\Xi\right\rangle \leq\left\langle \mathcal{F}_{N}(\rho),\Xi\right\rangle .
\]
\item [{$(2_{\tau}^{\Delta})$}] For every $\varphi$-stable $\Delta$-filtration
$\Xi$ on $V_{K}(\tau)$, 
\[
\left\langle \mathcal{F}_{H}(\tau),\Xi\right\rangle \leq\left\langle \mathcal{F}_{N}(\tau),\Xi\right\rangle .
\]
\item [{$(5)$}] $\mathcal{F}_{HN}=0$.
\item [{$(5_{\rho})$}] For every $\rho\in\Rep(G)$, $\mathcal{F}_{HN}(\rho)=0$.
\item [{$(5_{\tau})$}] $\mathcal{F}_{HN}(\tau)=0$.
\end{lyxlist}
\end{thm}
\begin{proof}
The following implications are obvious, or given by the previous lemma:
\[
\xymatrix{(1)\ar@{<=>}[r] & (1_{\rho})\ar@{=>}[r]\ar@{<=>}[d] & (1_{\tau})\ar@{<=>}[d]\\
 & (2_{\rho})\ar@{=>}[r]\ar@{<=>}[d] & (2_{\tau})\ar@{<=>}[d]\ar@{=>}[r] & (2^{\Delta})\\
(5)\ar@{<=>}[r] & (5_{\rho})\ar@{=>}[r] & (5_{\tau})
}
\]
We have $(2^{\Delta})\Leftrightarrow(2^{\mathbb{R}})\Leftrightarrow(2^{\mathbb{Q}[\Gamma]})$
as in \ref{subsec:2Dimplies2}, $(5_{\tau})\Rightarrow(5)$ by \cite[4.2.10 or 3.11.12]{Co14}
and finally $(2^{\mathbb{\mathbb{Q}}[\Gamma]})\Rightarrow(5_{\tau})$
since $(2^{\mathbb{Q}[\Gamma]})$ with $\Xi=\mathcal{F}_{HN}$ gives
\[
\left\langle \mathcal{F}_{H}(\tau),\mathcal{F}_{HN}(\tau)\right\rangle \leq\left\langle \mathcal{F}_{N}(\tau),\mathcal{F}_{HN}(\tau)\right\rangle 
\]
which implies $\mathcal{F}_{HN}(\tau)=0$ by the previous lemma.
\end{proof}

\subsection{~}

The above proof of $(2^{\Delta})\Rightarrow(1)$ looks much easier,
but it still uses Totaro's theorem to the effect that the Harder-Narasimhan
filtration is compatible with $\otimes$-products. Conversely, the
above theorem implies that weakly admissible objects are stable under
$\otimes$-product: apply $(1_{\tau})\Rightarrow(1_{\rho})$ with
$G=GL(V_{1})\times GL(V_{2})$, $\tau=\rho_{1}\boxplus\rho_{2}$ and
$\rho=\rho_{1}\boxtimes\rho_{2}$ with $\rho_{i}$ the tautological
representation of $GL(V_{i})$. In particular, Theorem~\ref{thm:Main1}
yields another proof of Totaro's result when $L=K$, which was of
course already known to Laffaille: if $M_{i}$ is a strongly divisible
lattice in $V_{i}$, then $M_{1}\otimes M_{2}$ is a strongly divisible
lattice in $V_{1}\otimes V_{2}$.

\section{Weakly admissible filtered isocrystals\label{sec:FarguesFilt}}

In this section, we define the Fargues $\mathbb{Q}$-filtration on
weakly admissible $\Gamma$-filtered isocrystals over $L$ (\ref{subsec:DefOfFarguesFiltr}),
show that it is compatible with tensor products (Theorem~\ref{thm:TensorProductFargues}),
and compute it as a convex projection~(Lemma~\ref{lem:CaractWeaklAdmFiltVectorCase}
and Proposition~\ref{prop:FarguesIsConvProj}).

\subsection{~\label{subsec:DefOfFarguesFiltr}}

Weakly admissible $\Gamma$-filtered isocrystals over $L$ have a
Harder-Narasimhan filtration of their own, defined by Fargues in~\cite[\S 9]{Fa12}
when $\Gamma=\mathbb{Z}$, for the slope function $\mu=\deg/\dim$,
where the degree function is now given by 
\[
\deg(V,\varphi,\mathcal{F})=-\deg(\mathcal{F}_{H})=-\deg(\mathcal{F}_{N})=\deg(\mathcal{F}_{N}^{\iota})\in\mathbb{Q}.
\]
An object $(V,\varphi,\mathcal{F})$ in $\Fil_{L}^{\Gamma}\Iso(k)^{wa}$
is thus again equipped with four filtrations: the \emph{Hodge} $\Gamma$-filtration
$\mathcal{F}=\mathcal{F}_{H}$ of $V_{L}$ by $L$-subspaces, the
pair of opposed \emph{Newton} $\mathbb{Q}$-filtrations $(\mathcal{F}_{N},\mathcal{F}_{N}^{\iota})$
of $V$ by $\varphi$-stable $K$-subspaces, and the \emph{Fargues
}$\mathbb{Q}$-filtration $\mathcal{F}_{F}$ of $V$ by ($\varphi$-stable)
weakly admissible $K$-subspaces. Note that the previous \emph{Harder-Narasimhan
}$\Gamma$-filtration $\mathcal{F}_{HN}$ on $V$ by $\varphi$-stable
$K$-subspaces is now trivial!

\subsection{~}

For a weakly admissible object $(V,\varphi,\mathcal{F})$ in $\Fil_{L}^{\Gamma}\Iso(k)$
and a subgroup $\Delta\neq0$ of $\mathbb{R}$, we denote by $\mathbf{F}^{\Delta}(V)$
the set of all $\Delta$-filtrations on $V$, by $\mathbf{F}^{\Delta}(V,\varphi)\subset\mathbf{F}^{\Delta}(V)$
its subset of $\Delta$-filtrations by $\varphi$-stable $K$-subspaces,
and by $\mathbf{F}^{\Delta}(V,\varphi,\mathcal{F})\subset\mathbf{F}^{\Delta}(V,\varphi)$
its subset of $\Delta$-filtrations by weakly admissible $\varphi$-stable
$K$-subspaces. Moreover, we equip $\mathbf{F}^{\mathbb{R}}(V)$ with
the $\varphi$-invariant CAT(0)-distance \cite[4.2.10]{Co14} defined
by 
\[
d(\mathcal{F}_{1},\mathcal{F}_{2})=\sqrt{\left\Vert \mathcal{F}_{1}\right\Vert ^{2}+\left\Vert \mathcal{F}_{2}\right\Vert ^{2}-2\left\langle \mathcal{F}_{1},\mathcal{F}_{2}\right\rangle }.
\]
Then $\mathbf{F}^{\mathbb{R}}(V,\varphi)$ and $\mathbf{F}^{\mathbb{R}}(V,\varphi,\mathcal{F})$
are closed and convex subsets of $\mathbf{F}^{\mathbb{R}}(V)$. In
fact:
\begin{lem}
\label{lem:CaractWeaklAdmFiltVectorCase}For $(V,\varphi,\mathcal{F})\in\Fil_{L}^{\Gamma}\Iso(k)^{wa}$
and a subgroup $0\neq\Delta\subset\mathbb{R}$ as above, 
\[
\mathbf{F}^{\Delta}(V,\varphi,\mathcal{F})=\left\{ \Xi\in\mathbf{F}^{\Delta}(V,\varphi):\left\langle \mathcal{F}_{H},\Xi\right\rangle =\left\langle \mathcal{F}_{N},\Xi\right\rangle \right\} .
\]
For $\Delta=\mathbb{R}$, this is a closed convex subset of $\mathbf{F}^{\mathbb{R}}(V)$.
Moreover $\mathcal{F}_{F}\in\mathbf{F}^{\mathbb{Q}}(V,\varphi,\mathcal{F})$
is the image of $\mathcal{F}_{N}^{\iota}\in\mathbf{F}^{\mathbb{Q}}(V,\varphi)$
under the convex projection $p:\mathbf{F}^{\mathbb{R}}(V,\varphi)\rightarrow\mathbf{F}^{\mathbb{R}}(V,\varphi,\mathcal{F})$.
\end{lem}
\begin{proof}
For $\Xi\in\mathbf{F}^{\Delta}(V,\varphi)$, set $\{\delta\in\Delta:\Gr_{\Xi}^{\delta}\neq0\}=\{\delta_{1}<\cdots<\delta_{n}\}$.
Then $\Xi^{\delta_{i}}$ is a $\varphi$-stable $K$-subspace of $V$.
Put $d_{H}(i)=\deg(\mathcal{F}_{H}\vert\Xi_{K}^{\delta_{i}})$, $d_{N}(i)=\deg(\mathcal{F}_{N}\vert\Xi^{\delta_{i}})$.
Then 
\[
\left\langle \mathcal{F}_{H},\Xi\right\rangle -\left\langle \mathcal{F}_{N},\Xi\right\rangle =\sum_{i=1}^{n}\delta_{i}\cdot\left(\Delta_{H}(i)-\Delta_{N}(i)\right)=\sum_{i=1}^{n}\Delta_{i}\cdot(d_{H}(i)-d_{N}(i))
\]
where $\Delta_{\star}(i)=d_{\star}(i)-d_{\star}(i+1)$ for $1\leq i<n$,
$\Delta_{\star}(n)=d_{\star}(n)$, $\Delta_{i}=\delta_{i}-\delta_{i-1}$
for $1<i\leq n$ and $\Delta_{1}=\delta_{1}$, as in the proof of
lemma~\ref{lem:DefWeakAdmEquScProd}. We now know that $d_{H}(i)\leq d_{N}(i)$
for all $i$ with equality for $i=1$. Since $\Delta_{i}>0$ for $i>1$,
we obtain 
\[
\left\langle \mathcal{F}_{H},\Xi\right\rangle =\left\langle \mathcal{F}_{N},\Xi\right\rangle \iff\forall1\leq i\leq n:\quad d_{H}(i)=d_{N}(i).
\]
The displayed equality immediately follows, which may also be written
as
\begin{eqnarray*}
\mathbf{F}^{\Delta}(V,\varphi,\mathcal{F}) & = & \left\{ \Xi\in\mathbf{F}^{\Delta}(V,\varphi):\left\langle \mathcal{F}_{H},\Xi\right\rangle +\left\langle \mathcal{F}_{N}^{\iota},\Xi\right\rangle =0\right\} ,\\
 & = & \left\{ \Xi\in\mathbf{F}^{\Delta}(V,\varphi):\left\langle \mathcal{F}_{H},\Xi\right\rangle +\left\langle \mathcal{F}_{N}^{\iota},\Xi\right\rangle \geq0\right\} .
\end{eqnarray*}
Since the scalar product $\left\langle -,-\right\rangle $ is continuous
and concave (see \cite[4.2.10]{Co14}), this is (for $\Delta=\mathbb{R}$)
a closed and convex subset of the closed and convex subset $\mathbf{F}^{\mathbb{R}}(V,\varphi)$
of the CAT(0)-space $\mathbf{F}^{\mathbb{R}}(V)$. Take now $\Xi=p(\mathcal{F}_{N}^{\iota})$,
the convex projection of $\mathcal{F}_{N}^{\iota}$ to $\mathbf{F}^{\mathbb{R}}(V,\varphi,\mathcal{F})$.
For $1\leq i\leq n$, a weakly admissible subspace $W$ of $\Gr_{\Xi}^{\delta_{i}}$
and a small $\epsilon>0$, let $\Xi'=\Xi(i,W,\epsilon)$ be the $\varphi$-stable
$\mathbb{R}$-filtration on $V$ with $\Gr_{\Xi'}^{\delta}=\Gr_{\Xi}^{\delta}$
unless $\delta=\delta_{i}$ or $\delta_{i}+\epsilon$, where $\Gr_{\Xi'}^{\delta_{i}}=\Gr_{\Xi}^{\delta_{i}}/W$
and $\Gr_{\Xi'}^{\delta_{i}+\epsilon}=W$. By definition of $\Xi$,
\[
\left\langle \Xi,\Xi\right\rangle -2\left\langle \mathcal{F}_{N}^{\iota},\Xi\right\rangle \leq\left\langle \Xi',\Xi'\right\rangle -2\left\langle \mathcal{F}_{N}^{\iota},\Xi'\right\rangle ,
\]
which easily unfolds to the following inequality, valid for small
$\epsilon>0$: 
\[
\epsilon^{2}\dim_{K}W+2\epsilon\left(\delta_{i}\dim_{K}W-\deg(\mathcal{F}_{N}^{\iota}\vert W)\right)\geq0.
\]
It follows that $\delta_{i}\geq\mu(W)$. For $W=\Gr_{\Xi}^{\delta_{i}}$,
we may also allow negative $\epsilon$'s with small absolute values
in the above argument, this time obtaining $\delta_{i}=\mu(\Gr_{\Xi}^{\delta_{i}})$.
Therefore $\Gr_{\Xi}^{\delta_{i}}$ is $\mu$-semi-stable of slope
$\delta_{i}$ for all $1\leq i\leq n$, and $\Xi$ thus equals $\mathcal{F}_{F}$. 
\end{proof}

\subsection{~}

Let now $G$ be a reductive group over $\mathbb{Q}_{p}$ and let
\[
(\varphi,\mathcal{F}):\Rep(G)\rightarrow\Fil_{L}^{\Gamma}\Iso(k)^{wa}
\]
be a weakly admissible $\Gamma$-filtered $G$-isocrystal over $L$.
We denote by 
\[
\mathbf{F}^{\Delta}(G_{K},\varphi,\mathcal{F})\subset\mathbf{F}^{\Delta}(G_{K},\varphi)\subset\mathbf{F}^{\Delta}(G_{K})
\]
the set of $\Delta$-filtrations on the fiber functor $V_{K}:\Rep(G)\rightarrow\Vect(K)$
by respectively $\varphi$-stable and $\varphi$-stable weakly admissible
$K$-subspaces. We also fix a faithful representation $\tau$ of $G$
and equip $\mathbf{F}^{\mathbb{R}}(G_{K})$ with the CAT(0)-metric
defined by 
\[
d_{\tau}(\mathcal{F}_{1},\mathcal{F}_{2})=\sqrt{\left\Vert \mathcal{F}_{1}\right\Vert _{\tau}^{2}+\left\Vert \mathcal{F}_{2}\right\Vert _{\tau}^{2}-2\left\langle \mathcal{F}_{1},\mathcal{F}_{2}\right\rangle _{\tau}}.
\]

\begin{lem}
\label{lem:CaractWeakAdmFiltrG}With notations as above, 
\[
\mathbf{F}^{\Delta}(G_{K},\varphi,\mathcal{F})=\left\{ \Xi\in\mathbf{F}^{\Delta}(G_{K},\varphi):\left\langle \mathcal{F}_{H},\Xi\right\rangle _{\tau}=\left\langle \mathcal{F}_{N},\Xi\right\rangle _{\tau}\right\} .
\]
For $\Delta=\mathbb{R}$, this is a closed convex subset of $(\mathbf{F}^{\mathbb{R}}(G_{K}),d_{\tau})$. 
\end{lem}
\begin{proof}
If $\Xi\in\mathbf{F}^{\Delta}(G_{K},\varphi,\mathcal{F})$, then $\Xi(\tau)\in\mathbf{F}^{\Delta}(V_{K}(\tau),\varphi(\tau),\mathcal{F}(\tau))$,
thus
\[
\left\langle \mathcal{F}_{H},\Xi\right\rangle _{\tau}=\left\langle \mathcal{F}_{H}(\tau),\Xi(\tau)\right\rangle =\left\langle \mathcal{F}_{N}(\tau),\Xi(\tau)\right\rangle =\left\langle \mathcal{F}_{N},\Xi\right\rangle _{\tau}
\]
by the previous lemma. Suppose conversely that $\Xi\in\mathbf{F}^{\Delta}(G_{K},\varphi)$
satisfies this equation, choose a splitting of $\Xi$ (which exists
by~\cite[3.11.3]{Co14}) and use it to transport $(\Gr_{\Xi}(\varphi),\Gr_{\Xi}(\mathcal{F}))$
back from the fiber functor $\Gr_{\Xi}(V_{K})$ to $(\varphi',\mathcal{F}')$
on $V_{K}$. Since $\Fil_{L}^{\Gamma}\Iso(k)^{wa}$ is a strictly
full subcategory of $\Fil_{L}^{\Gamma}\Iso(k)$ which is stable under
extensions, it is sufficient to establish that $(\varphi',\mathcal{F}')$
is weakly admissible, and this now follows from $(1_{\tau})\Rightarrow(1)$
of theorem~\ref{thm:Main1withL} since $(V_{K}(\tau),\varphi'(\tau),\mathcal{F}'(\tau))$
is weakly admissible by the previous lemma. For $\Delta=\mathbb{R}$,
one shows as above that 
\begin{eqnarray*}
\mathbf{F}^{\mathbb{R}}(G_{K},\varphi,\mathcal{F}) & = & \left\{ \Xi\in\mathbf{F}^{\mathbb{R}}(G_{K},\varphi):\left\langle \mathcal{F}_{H},\Xi\right\rangle _{\tau}+\left\langle \mathcal{F}_{N}^{\iota},\Xi\right\rangle _{\tau}=0\right\} \\
 & = & \left\{ \Xi\in\mathbf{F}^{\mathbb{R}}(G_{K},\varphi):\left\langle \mathcal{F}_{H},\Xi\right\rangle _{\tau}+\left\langle \mathcal{F}_{N}^{\iota},\Xi\right\rangle _{\tau}\geq0\right\} 
\end{eqnarray*}
is a closed convex subset of $\mathbf{F}^{\mathbb{R}}(G_{K},\varphi)$.
\end{proof}

\subsection{~}

We may now establish the analog of Totaro's tensor product theorem
for the Fargues filtration $\mathcal{F}_{F}$. In the classical setting
of $p$-adic Hodge theory where $\Gamma=\mathbb{Z}$ and $L$ is a
totally ramified finite algebraic extension of $K$, an entirely different
proof is given in~\cite[\S 9.2]{Fa12}, using Fontaine's functor
to work on the Galois side.
\begin{thm}
\label{thm:TensorProductFargues}The Fargues filtration defines a
$\otimes$-functor
\[
\mathcal{F}_{F}:\Fil_{L}^{\Gamma}\Iso(k)^{wa}\rightarrow\Fil_{K}^{\mathbb{Q}}.
\]
\end{thm}
\begin{proof}
The general formalism of Harder-Narasimhan filtrations yields the
functoriality of $\mathcal{F}_{F}$ and reduces its compatibility
with tensor products to the following statement: if $(V_{i},\varphi_{i},\mathcal{F}_{i})\in\Fil_{L}^{\Gamma}\Iso(k)^{wa}$
is semi-stable of slope $\mu_{i}\in\mathbb{Q}$ for $i\in\{1,2\}$,
\[
(V,\varphi,\mathcal{F})=(V_{1}\otimes V_{2},\varphi_{1}\otimes\varphi_{2},\mathcal{F}_{1}\otimes\mathcal{F}_{2})\in\Fil_{L}^{\Gamma}\Iso(k)^{wa}
\]
is semi-stable of slope $\mu=\mu_{1}+\mu_{2}$, which means that for
every weakly admissible subobject $X$ of $V$, $\deg(\mathcal{F}_{N}^{\iota}\vert X)\leq\deg(V(\mu)\vert X)$
with equality for $X=V$. Here $V(\nu)$ denotes the filtration on
$V$ such that $\Gr_{V(\nu)}^{\nu}=V$. Exactly as in Lemma~\ref{lem:DefWeakAdmEquScProd},
this condition is itself equivalent to the following one: for every
$\Xi\in\mathbf{F}^{\mathbb{R}}(V,\varphi,\mathcal{F})$, 
\[
\left\langle \mathcal{F}_{N}^{\iota},\Xi\right\rangle \leq\left\langle V(\mu),\Xi\right\rangle \quad(\mbox{or}:\,\left\langle \mathcal{F}_{N}^{\iota},\Xi\right\rangle +\left\langle V(-\mu),\Xi\right\rangle \leq0).
\]
We adapt Totaro's proof to this setting. Applying the functor $\mathbf{F}^{\mathbb{R}}(-)$
to
\[
GL(V_{1})\times GL(V_{2})\twoheadrightarrow G=GL(V_{1})\times GL(V_{2})/\Delta(\mathbb{G}_{m})\hookrightarrow GL(V_{1}\otimes V_{2})
\]
where $\Delta$ is the diagonal embedding yields a factorization
\[
\mathbf{F}^{\mathbb{R}}(V_{1})\times\mathbf{F}^{\mathbb{R}}(V_{2})\twoheadrightarrow\mathbf{F}^{\mathbb{R}}(G)\hookrightarrow\mathbf{F}^{\mathbb{R}}(V_{1}\otimes V_{2})
\]
of the map $(\mathcal{G}_{1},\mathcal{G}_{2})\mapsto\mathcal{G}_{1}\otimes\mathcal{G}_{2}$,
which identifies $\mathbf{F}^{\mathbb{R}}(G)$ with a closed and convex
subset of $\mathbf{F}^{\mathbb{R}}(V_{1}\otimes V_{2})$. Let $p:\mathbf{F}^{\mathbb{R}}(V_{1}\otimes V_{2})\twoheadrightarrow\mathbf{F}^{\mathbb{R}}(G)$
be the convex projection and pick $(\Xi_{1},\Xi_{2})\in\mathbf{F}^{\mathbb{R}}(V_{1})\times\mathbf{F}^{\mathbb{R}}(V_{2})$
mapping to $p(\Xi)\in\mathbf{F}^{\mathbb{R}}(G)$ (this is Kempf's
filtration of \cite[\S 2]{To96}). Since $p$ is $\varphi$-equivariant,
$p(\Xi)$ belongs to $\mathbf{F}^{\mathbb{R}}(G,\varphi)$ and then
also $\Xi_{i}\in\mathbf{F}^{\mathbb{R}}(V_{i},\varphi)$ for $i\in\{1,2\}$.
Moreover, for any $\left(\mathcal{G}_{1},\mathcal{G}_{2}\right)\in\mathbf{F}^{\mathbb{R}}(V_{1,L})\times\mathbf{F}^{\mathbb{R}}(V_{2,L})$,
\begin{equation}
\left\langle \mathcal{G}_{1}\otimes\mathcal{G}_{2},\Xi\right\rangle \leq\left\langle \mathcal{G}_{1}\otimes\mathcal{G}_{2},p(\Xi)\right\rangle \label{eq:projIncreaseScalar}
\end{equation}
by \cite[5.5.12]{Co14} since $p$ is also the restriction to $\mathbf{F}^{\mathbb{R}}(V_{1}\otimes V_{2})$
of the convex projection from $\mathbf{F}^{\mathbb{R}}(V_{1,L}\otimes V_{2,L})$
onto $\mathbf{F}^{\mathbb{R}}(G_{L})$. On the other hand, one checks
easily that
\begin{eqnarray}
\left\langle \mathcal{G}_{1}\otimes\mathcal{G}_{2},\mathcal{H}_{1}\otimes\mathcal{H}_{2}\right\rangle  & = & \left\langle \mathcal{G}_{1},\mathcal{H}_{1}\right\rangle \cdot\dim V_{2}+\left\langle \mathcal{G}_{2},\mathcal{H}_{2}\right\rangle \cdot\dim V_{1}\nonumber \\
 &  & +\deg\mathcal{G}_{1}\cdot\deg\mathcal{H}_{2}+\deg\mathcal{H}_{1}\cdot\deg\mathcal{G}_{2}\label{eq:ScalProdofTensor}
\end{eqnarray}
for every $(\mathcal{G}_{i},\mathcal{H}_{i})\in\mathbf{F}^{\mathbb{R}}(V_{i,L})$.
We will now apply these two formulas to 
\begin{eqnarray*}
\mathcal{F}_{N}^{\iota}(\mbox{on }V) & = & \mathcal{F}_{N}^{\iota}(\mbox{on }V_{1})\otimes\mathcal{F}_{N}^{\iota}(\mbox{on }V_{2}),\\
\mathcal{F}_{H}(\mbox{on }V_{L}) & = & \mathcal{F}_{H}(\mbox{on }V_{1})\otimes\mathcal{F}_{H}(\mbox{on }V_{2}),\\
V(-\mu) & = & V_{1}(-\mu_{1})\otimes V_{2}(-\mu_{2}),\\
p(\Xi) & = & \Xi_{1}\otimes\Xi_{2}.
\end{eqnarray*}
First recall that since $(V,\varphi,\mathcal{F})$ is weakly admissible
(by Totaro's theorem!), 
\[
\left\langle \mathcal{F}_{H},p(\Xi)\right\rangle +\left\langle \mathcal{F}_{N}^{\iota},p(\Xi)\right\rangle \leq0.
\]
Since $\left\langle \mathcal{F}_{H},\Xi\right\rangle +\left\langle \mathcal{F}_{N}^{\iota},\Xi\right\rangle =0$
by assumption on $\Xi$ and Lemma~\ref{lem:CaractWeaklAdmFiltVectorCase},
actually 
\[
\left\langle \mathcal{F}_{H},p(\Xi)\right\rangle +\left\langle \mathcal{F}_{N}^{\iota},p(\Xi)\right\rangle =0
\]
by~(\ref{eq:projIncreaseScalar}). Applying formula (\ref{eq:ScalProdofTensor})
twice and grouping terms, we obtain 
\begin{eqnarray*}
0 & = & \left(\left\langle \mathcal{F}_{H},\Xi_{1}\right\rangle +\left\langle \mathcal{F}_{N}^{\iota},\Xi_{1}\right\rangle \right)\cdot\dim V_{2}\\
 & + & \left(\left\langle \mathcal{F}_{H},\Xi_{2}\right\rangle +\left\langle \mathcal{F}_{N}^{\iota},\Xi_{2}\right\rangle \right)\cdot\dim V_{1}\\
 & + & \left(\deg(\mathcal{F}_{H}\mbox{ on }V_{1})+\deg(\mathcal{F}_{N}^{\iota}\,\mbox{on }V_{1})\right)\cdot\deg\Xi_{2}\\
 & + & \left(\deg(\mathcal{F}_{H}\mbox{ on }V_{2})+\deg(\mathcal{F}_{N}^{\iota}\,\mbox{on }V_{2})\right)\cdot\deg\Xi_{1}.
\end{eqnarray*}
Since $V_{1}$ and $V_{2}$ are weakly admissible, the first two terms
are $\leq0$ and the last two trivial. Thus $\left\langle \mathcal{F}_{H},\Xi_{i}\right\rangle +\left\langle \mathcal{F}_{N}^{\iota},\Xi_{i}\right\rangle =0$
for $i\in\{1,2\}$, i.e.~$\Xi_{i}\in\mathbf{F}^{\mathbb{R}}(V_{i},\varphi_{i},\mathcal{F}_{i})$.
Now applying our two formulas (\ref{eq:projIncreaseScalar}) and (\ref{eq:ScalProdofTensor})
twice again also gives that
\begin{eqnarray*}
\left\langle \mathcal{F}_{N}^{\iota},\Xi\right\rangle +\left\langle V(-\mu),\Xi\right\rangle  & \leq & \left\langle \mathcal{F}_{N}^{\iota},p(\Xi)\right\rangle +\left\langle V(-\mu),p(\Xi)\right\rangle \\
 & \leq & \left(\left\langle \mathcal{F}_{N}^{\iota},\Xi_{1}\right\rangle +\left\langle V_{1}(-\mu_{1}),\Xi_{1}\right\rangle \right)\cdot\dim V_{2}\\
 & + & \left(\left\langle \mathcal{F}_{N}^{\iota},\Xi_{2}\right\rangle +\left\langle V_{2}(-\mu_{1}),\Xi_{2}\right\rangle \right)\cdot\dim V_{1}\\
 & + & \left(\deg(\mathcal{F}_{N}^{\iota}\mbox{ on }V_{1})+\deg(V_{1}(-\mu_{1}))\right)\cdot\deg\Xi_{2}\\
 & + & \left(\deg(\mathcal{F}_{N}^{\iota}\mbox{ on }V_{2})+\deg(V_{2}(-\mu_{2}))\right)\cdot\deg\Xi_{1}.
\end{eqnarray*}
Since $V_{i}$ is semi-stable of slope $\mu_{i}$ for $i\in\{1,2\}$,
the first two terms are $\leq0$ and the last two trivial. Thus~$\left\langle \mathcal{F}_{N}^{\iota},\Xi\right\rangle +\left\langle V(-\mu),\Xi\right\rangle \leq0$
for all $\Xi\in\mathbf{F}^{\mathbb{R}}(V,\varphi,\mathcal{F})$ and
$(V,\varphi,\mathcal{F})$ is indeed semi-stable of slope $\mu$.
\end{proof}

\subsection{~}

A weakly admissible $\Gamma$-filtered $G$-isocrystal $(\varphi,\mathcal{F})$
over $L$ thus again yields four filtrations: the \emph{Hodge }$\Gamma$-filtration
$\mathcal{F}_{H}=\mathcal{F}\in\mathbf{F}^{\Gamma}(G_{L})$, the opposed
\emph{Newton} $\mathbb{Q}$-filtrations $\mathcal{F}_{N},\mathcal{F}_{N}^{\iota}\in\mathbf{F}^{\mathbb{Q}}(G_{K},\varphi)$
and the Fargues $\mathbb{Q}$-filtration $\mathcal{F}_{F}\in\mathbf{F}^{\mathbb{Q}}(G_{K},\varphi,\mathcal{F})$,
which is obtained by composing the exact $\otimes$-functor $(\varphi,\mathcal{F})$
with the $\otimes$-functor of the previous proposition (the resulting
$\otimes$-functor is exact because it is plainly additive and every
exact sequence in $\Rep(G)$ is split). The former Harder-Narasimhan
filtration $\mathcal{F}_{HN}$ of section~\ref{sec:EquivForWeaklyAdm}
is trivial. We claim that, just as in the $GL(V)$-case:
\begin{prop}
\label{prop:FarguesIsConvProj}The Fargues $\mathbb{Q}$-filtration
$\mathcal{F}_{F}$ is the convex projection of $\mathcal{F}_{N}^{\iota}$.
\end{prop}
\begin{proof}
The map $\mathcal{G}\mapsto\mathcal{G}(\tau)$ identifies $\mathbf{F}^{\mathbb{R}}(G_{K})$
with a $\varphi$-stable closed convex subset of $\mathbf{F}^{\mathbb{R}}(V_{K}(\tau))$.
Let $p$ be the convex projection from $\mathbf{F}^{\mathbb{R}}(V_{K}(\tau))$
to $\mathbf{F}^{\mathbb{R}}(G_{K})$. It is $\varphi$-equivariant
and thus maps $\mathbf{F}^{\mathbb{R}}(V_{K}(\tau),\varphi(\tau))$
to $\mathbf{F}^{\mathbb{R}}(G_{K},\varphi)$. Moreover, 
\[
\left\langle \mathcal{H}(\tau),\mathcal{G}\right\rangle \leq\left\langle \mathcal{H}(\tau),p(\mathcal{G})(\tau)\right\rangle =\left\langle \mathcal{H},p(\mathcal{G})\right\rangle _{\tau}
\]
for every $\mathcal{H}\in\mathbf{F}^{\mathbb{R}}(G_{K})$ (or $\mathbf{F}^{\mathbb{R}}(G_{L})$)
and $\mathcal{G}\in\mathbf{F}^{\mathbb{R}}(V_{K}(\tau))$, exactly
as in the proof of the previous proposition. Suppose now that $\mathcal{G}$
belongs to $\mathbf{F}^{\mathbb{R}}(V_{K}(\tau),\varphi(\tau),\mathcal{F}(\tau))$.
Then $p(\mathcal{G})$ also belongs to $\mathbf{F}^{\mathbb{R}}(G_{K},\varphi,\mathcal{F})$.
Indeed, since $(\varphi,\mathcal{F}$) is weakly admissible, 
\[
\left\langle \mathcal{F}_{H},p(\mathcal{G})\right\rangle _{\tau}+\left\langle \mathcal{F}_{N}^{\iota},p(\mathcal{G})\right\rangle _{\tau}\leq0
\]
but also $\left\langle \mathcal{F}_{H}(\tau),\mathcal{G}\right\rangle +\left\langle \mathcal{F}_{N}^{\iota}(\tau),\mathcal{G}\right\rangle =0$
thus actually
\[
\left\langle \mathcal{F}_{H},p(\mathcal{G})\right\rangle _{\tau}+\left\langle \mathcal{F}_{N}^{\iota},p(\mathcal{G})\right\rangle _{\tau}=0
\]
i.e.~$p(\mathcal{G})$ belongs to $\mathbf{F}^{\mathbb{R}}(G_{K},\varphi,\mathcal{F})$.
Since $p$ is non-expanding and $\mathcal{F}_{N}^{\iota}\in\mathbf{F}^{\mathbb{R}}(G_{K})$,
{\small{}
\[
d\left(\mathcal{F}_{N}^{\iota}(\tau),\mathcal{G}\right)\geq d\left(\mathcal{F}_{N}^{\iota}(\tau),p(\mathcal{G})(\tau)\right)=d_{\tau}\left(\mathcal{F}_{N}^{\iota},p(\mathcal{G})\right)\geq d_{\tau}(\mathcal{F}_{N}^{\iota},\mathcal{F}'_{F})=d(\mathcal{F}_{N}^{\iota}(\tau),\mathcal{F}'_{F}(\tau))
\]
}where $\mathcal{F}'_{F}$ is the convex projection of $\mathcal{F}_{N}^{\iota}\in\mathbf{F}^{\mathbb{R}}(G_{K},\varphi)$
to $\mathbf{F}^{\mathbb{R}}(G_{K},\varphi,\mathcal{F})$. Since this
holds true for any $\mathcal{G}\in\mathbf{F}^{\mathbb{R}}(V_{K}(\tau),\varphi(\tau),\mathcal{F}(\tau))$,
it follows that $\mathcal{F}'_{F}(\tau)=\mathcal{F}_{F}(\tau)$ by
Lemma~\ref{lem:CaractWeaklAdmFiltVectorCase}. But then $\mathcal{F}'_{F}=\mathcal{F}_{F}$
since $\mathbf{F}^{\mathbb{R}}(G_{K})\hookrightarrow\mathbf{F}^{\mathbb{R}}(V_{K}(\tau))$
is injective.
\end{proof}
\bibliographystyle{plain}
\bibliography{MyBib}

\begin{thebibliography}{10}

\bibitem{BrHa99}
M.~R. Bridson and A.~Haefliger.
\newblock {\em {Metric spaces of non-positive curvature}}, volume 319 of {\em
  {Grundlehren der Mathematischen Wissenschaften [Fundamental Principles of
  Mathematical Sciences]}}.
\newblock Springer-Verlag, Berlin, 1999.

\bibitem{ChPh06}
P.~Chaoha and A.~Phon-on.
\newblock {A note on fixed point sets in {CAT}(0) spaces}.
\newblock {\em J. Math. Anal. Appl.}, 320(2):983--987, 2006.

\bibitem{ChVi18}
Miaofen Chen and Eva Viehmann.
\newblock Affine {D}eligne-{L}usztig varieties and the action of {$J$}.
\newblock {\em J. Algebraic Geom.}, 27(2):273--304, 2018.

\bibitem{CoFo00}
Pierre Colmez and Jean-Marc Fontaine.
\newblock {Construction des repr{\'e}sentations {$p$}-adiques semi-stables}.
\newblock {\em Invent. Math.}, 140(1):1--43, 2000.

\bibitem{Co14}
Christophe Cornut.
\newblock {Filtrations and Buildings. To appear in Memoirs of the AMS.}

\bibitem{Co13}
Christophe Cornut.
\newblock {A fixed point theorem in {E}uclidean buildings}.
\newblock {\em Adv. Geom.}, 16(4):487--496, 2016.

\bibitem{CoNi16}
Christophe Cornut and Marc-Hubert Nicole.
\newblock {Cristaux et immeubles}.
\newblock {\em Bull. Soc. Math. France}, 144(1):125--143, 2016.

\bibitem{DaOrRa10}
J.-F. Dat, S.~Orlik, and M.~Rapoport.
\newblock {\em {Period domains over finite and {$p$}-adic fields}}, volume 183
  of {\em {Cambridge Tracts in Mathematics}}.
\newblock Cambridge University Press, Cambridge, 2010.

\bibitem{Fa95}
Gerd Faltings.
\newblock {Mumford-{S}tabilit{\"a}t in der algebraischen {G}eometrie}.
\newblock In {\em {Proceedings of the {I}nternational {C}ongress of
  {M}athematicians, {V}ol.\ 1, 2 ({Z}{\"u}rich, 1994)}}, pages 648--655.
  Birkh{\"a}user, Basel, 1995.

\bibitem{Fa12}
L.~Fargues.
\newblock {Th{\'e}orie de la r{\'e}duction pour les groupes p-divisibles,
  Preprint}.

\bibitem{FoLa82}
Jean-Marc Fontaine and Guy Laffaille.
\newblock {Construction de repr{\'e}sentations {$p$}-adiques}.
\newblock {\em Ann. Sci. {\'E}cole Norm. Sup. (4)}, 15(4):547--608 (1983),
  1982.

\bibitem{FoRa05}
Jean-Marc Fontaine and Michael Rapoport.
\newblock {Existence de filtrations admissibles sur des isocristaux}.
\newblock {\em Bull. Soc. Math. France}, 133(1):73--86, 2005.

\bibitem{Ga10}
Q{\"e}ndrim~R. Gashi.
\newblock {On a conjecture of {K}ottwitz and {R}apoport}.
\newblock {\em Ann. Sci. {\'E}c. Norm. Sup{\'e}r. (4)}, 43(6):1017--1038, 2010.

\bibitem{Ka06}
Michael Kapovich.
\newblock Generalized triangle inequalities and their applications.
\newblock In {\em International {C}ongress of {M}athematicians. {V}ol. {II}},
  pages 719--741. Eur. Math. Soc., Z\"{u}rich, 2006.

\bibitem{Ko85}
R.~E. Kottwitz.
\newblock {Isocrystals with additional structure}.
\newblock {\em Compositio Math.}, 56(2):201--220, 1985.

\bibitem{Ko03}
Robert~E. Kottwitz.
\newblock {On the {H}odge-{N}ewton decomposition for split groups}.
\newblock {\em Int. Math. Res. Not.}, (26):1433--1447, 2003.

\bibitem{Ku14}
Shrawan Kumar.
\newblock A survey of the additive eigenvalue problem.
\newblock {\em Transform. Groups}, 19(4):1051--1148, 2014.
\newblock With an appendix by M. Kapovich.

\bibitem{Laf80}
Guy Laffaille.
\newblock {Groupes {$p$}-divisibles et modules filtr{\'e}s: le cas peu
  ramifi{\'e}}.
\newblock {\em Bull. Soc. Math. France}, 108(2):187--206, 1980.

\bibitem{La00}
E.~Landvogt.
\newblock {Some functorial properties of the {B}ruhat-{T}its building}.
\newblock {\em J. Reine Angew. Math.}, 518:213--241, 2000.

\bibitem{RaRi96}
M.~Rapoport and M.~Richartz.
\newblock {On the classification and specialization of {$F$}-isocrystals with
  additional structure}.
\newblock {\em Compositio Math.}, 103(2):153--181, 1996.

\bibitem{RaZi96}
M.~Rapoport and Th. Zink.
\newblock {\em {Period spaces for {$p$}-divisible groups}}, volume 141 of {\em
  {Annals of Mathematics Studies}}.
\newblock Princeton University Press, Princeton, NJ, 1996.

\bibitem{RaZi02}
M.~Rapoport and Th. Zink.
\newblock {A finiteness theorem in the {B}ruhat-{T}its building: an application
  of {L}andvogt's embedding theorem}.
\newblock {\em Indag. Math. (N.S.)}, 10(3):449--458, 1999.

\bibitem{Ti79}
Jacques Tits.
\newblock {Reductive groups over local fields}.
\newblock In {\em {Automorphic forms, representations and {$L$}-functions
  ({P}roc. {S}ympos. {P}ure {M}ath., {O}regon {S}tate {U}niv., {C}orvallis,
  {O}re., 1977), {P}art 1}}, {Proc. Sympos. Pure Math., XXXIII}, pages 29--69.
  Amer. Math. Soc., Providence, R.I., 1979.

\bibitem{To96}
Burt Totaro.
\newblock {Tensor products in {$p$}-adic {H}odge theory}.
\newblock {\em Duke Math. J.}, 83(1):79--104, 1996.

\bibitem{VoWe11}
I.~Vollaard and T.~Wedhorn.
\newblock {The supersingular locus of the {S}himura variety of {GU}$(1,n-1)$
  {II}}.
\newblock {\em Inventiones Math.}, (184):591--627, 2011.

\end{thebibliography}

\end{document}